\documentclass[10pt,reqno]{amsart}

\usepackage{amsthm,amsmath,amstext,amssymb,amscd,euscript, mathrsfs, dsfont,multicol,times,enumerate,subfig,sidecap}
\usepackage{enumerate}
\usepackage{amsmath, amssymb, amsthm}
\usepackage{mathrsfs}
\usepackage{esint}
\usepackage{xcolor}
\usepackage{mathtools}
\usepackage{tikz}

\usepackage{bm}

\usepackage{bbm}

\newtheorem{thm}{Theorem}[section]

\newtheorem{lem}[thm]{Lemma}
\newtheorem{prop}[thm]{Proposition}
\theoremstyle{definition}
\newtheorem{defn}[thm]{Definition}

\newtheorem{rem}[thm]{Remark}

\newcommand\bI{\mathbbm{1}}

\newcommand\bN{\mathbb{N}}

\newcommand\bQ{\mathbb{Q}}
\newcommand\bR{\mathbb{R}}
\newcommand\bS{\mathbb{S}}

\newcommand\bZ{\mathbb{Z}}

\newcommand\cB{\mathcal{B}}

\newcommand\cD{\mathcal{D}}
\newcommand\cF{\mathcal{F}}

\newcommand\cI{\mathcal{I}}

\newcommand\cP{\mathcal{P}}

\newcommand\cS{\mathcal{S}}
\newcommand\cM{\mathcal{M}}

\newcommand\cQ{\mathcal{Q}}
\newcommand\cX{\mathcal{X}}

\newcommand{\dd}{\,\mathrm{d}}

\DeclareMathOperator*{\esssup}{ess\,sup}

\newcommand{\supp}{\mathrm{supp}}

\newcommand{\aint}{-\hspace{-0.40cm}\int}

\newcommand{\mysection}[1]{\section{#1}
\setcounter{equation}{0}}

\def\XXint#1#2#3{{\setbox0=\hbox{$#1{#2#3}{\int}$ }
\vcenter{\hbox{$#2#3$ }}\kern-.58\wd0}}

\makeatletter
\def\@tocline#1#2#3#4#5#6#7{\relax
  \ifnum #1>\c@tocdepth % then omit
  \else
    \par \addpenalty\@secpenalty\addvspace{#2}%
    \begingroup \hyphenpenalty\@M
    \@ifempty{#4}{%
      \@tempdima\csname r@tocindent\number#1\endcsname\relax
    }{%
      \@tempdima#4\relax
    }%
    \parindent\z@ \leftskip#3\relax \advance\leftskip\@tempdima\relax
    \rightskip\@pnumwidth plus4em \parfillskip-\@pnumwidth
    #5\leavevmode\hskip-\@tempdima
      \ifcase #1
       \or\or \hskip 1em \or \hskip 2em \else \hskip 3em \fi%
      #6\nobreak\relax
    \dotfill\hbox to\@pnumwidth{\@tocpagenum{#7}}\par
    \nobreak
    \endgroup
  \fi}
\makeatother

\begin{document}
\title[Characterizations of weighted Besov and Triebel-Lizorkin spaces with variable smoothness]
{Characterizations of weighted Besov and Triebel-Lizorkin spaces with variable smoothness}



\thanks{}

\maketitle
\begin{center}
\vspace{-1em}
\small
\textsc{Jae-Hwan Choi}\\
School of Mathematics, Korea Institute for Advanced Study, 85 Hoegiro Dongdaemun-gu, Seoul 02455, Republic of Korea\\
\texttt{jhchoi@kias.re.kr}

\vspace{0.5em}
\textsc{Jin Bong Lee}\\
Research Institute of Mathematics, Seoul National University, Gwanak-ro 1, Gwanak-gu, Seoul 08826, Republic of Korea\\
\texttt{jinblee@snu.ac.kr}

\vspace{0.5em}
\textsc{Jinsol Seo}\\
School of Mathematics, Korea Institute for Advanced Study, 85 Hoegiro Dongdaemun-gu, Seoul 02455, Republic of Korea\\
\texttt{seo9401@kias.re.kr}

\vspace{0.5em}
\textsc{Kwan Woo}\\
Departement Mathematik und Informatik, Universit\"at Basel, CH-4051 Basel, Switzerland\\
\texttt{kwan.woo@unibas.ch}
\vspace{1em}
\end{center}

\begin{abstract}
In this paper, we study different types of weighted Besov and Triebel-Lizorkin spaces with variable smoothness. 
The function spaces can be defined by means of the Littlewood-Paley theory in the field of Fourier analysis,
while there are other norms arising in the theory of partial differential equations such as Sobolev-Slobodeckij spaces.
It is known that two norms are equivalent when one considers constant regularity function spaces without weights.
We show that the equivalence still holds for variable smoothness and weights, which is accomplished by using shifted maximal functions, Peetre's maximal functions, and the reverse H\"older inequality.
Moreover, we obtain a weighted regularity estimate for time-fractional evolution equations and a generalized Sobolev embedding theorem without weights.
\end{abstract}

\mysection{Introduction and main results}

Besov and Triebel-Lizorkin spaces have been crucial tools for measuring the smoothness of functions in the theory of partial differential equations(PDEs) and approximation theory.
Due to their significance, the \emph{characterization} of these spaces---namely, the study of several equivalent norms or representations describing the same function space---has also become an important topic in both PDE theory and harmonic analysis. 

For instance, two representative types of norms are frequently and separately used in Fourier analysis and in PDE theory defined by the Littlewood--Paley decomposition and by difference operators, respectively.
Let us consider norms of Besov spaces $B_{2,2}^s(\bR^d)$, which coincides with the Sobolev space $H^s(\mathbb{R}^d)$.
Then, it follows from the Littlewood-Paley theory that
\begin{align}\label{Besov_norm_lp}
	\|f\|_{B_{2,2}^s(\bR^d)} = \|S_0 f\|_{L_2(\bR^d)} + \left( \sum_{j\geq0} 2^{2sj} \| \Delta_j f\|_{L_2(\bR^d)}^2 \right)^{1/2},
\end{align}
where $S_0$ and $\Delta_j$ are the Littlewood-Paley projection operators whose definitions are given by \eqref{ineq_230315_1246}.
It is well-known the quantity \eqref{Besov_norm_lp} is equivalent to 
\begin{align}\label{Besov_norm_diff1}
	\|f \|_{W_{2}^{n}(\bR^d)} +\left( \int_{|h| \leq 1} \left| \frac{\left\|\mathcal{D}_{h}^2 \left(f^{(n)}\right)\right\|_{L_2(\mathbb{R}^d)} }{|h|^\alpha} \right|^2 ~\frac{\mathrm{d}h}{|h|^d}\right)^{1/2},
\end{align}
where $s = n+\alpha$, $n\in\mathbb{Z}$, $\alpha\in[0,1)$, and $f^{(n)}$ denotes the $n$-th derivative of $f$.
Note that $\mathcal{D}_h$ is a difference operator given by $\cD_h f (x) = f(x+h) - f(x)$ and $\cD_h^L = \cD_h \cD_h^{L-1}$.

The equivalence between \eqref{Besov_norm_lp} and \eqref{Besov_norm_diff1} for unweighted spaces are well known, \textit{e.g.}, \cite{triebel1983theory}. 
Later, Kaljabin and Lizorkin \cite{Lizorkin1987generalized} extended these results to the case of variable smoothness, where the dyadic factor $2^j$ and the spatial scale $|h|^{\alpha}$ in \eqref{Besov_norm_lp} and \eqref{Besov_norm_diff1} are replaced by general functions $\phi(2^j)$ and $\phi(|h|^{-1})$, respectively.
These works illustrate how different analytic representations of a space lead to equivalent descriptions of regularity, and thus play a foundational role in the development of modern function space theory. 

However, when extending the setting to \emph{weighted} spaces, $L_2(\mathbb{R}^d, w\,\mathrm{d}x)$, the equivalence of norms such as \eqref{Besov_norm_lp} and \eqref{Besov_norm_diff1} remains far less understood, despite their fundamental importance in applications to PDEs and harmonic analysis. 
Weighted characterizations of Besov and Triebel--Lizorkin spaces have been studied extensively for certain types of norms: 
H.-Q. Bui \cite{Bui1982, Bui1984} and H.-Q. Bui, M. Paluszy\'nski, M. H. Taibleson \cite{BPT1996} obtained characterizations for $A_\infty$ weights using heat kernels and smooth kernel representations, and more recently H.-Q. Bui, T. A. Bui, and X. T. Duong \cite{BuiDuong_2020} extended these to spaces of homogeneous type.
In a different direction, L.I. Hedberg and Y. Netrusov \cite{HN07} provided a unified axiomatic framework encompassing a broad class of function and distribution spaces, including the scales of Besov and Triebel-Lizorkin spaces with variable smoothness. 
However, the characterization of \cite{HN07} is clearly different from that of this paper.
To the best of our knowledge, the equivalence of norms of types between \eqref{Besov_norm_lp} and \eqref{Besov_norm_diff1} in the weighted framework has remained open. 
One of the main motivations of this paper is to fill this fundamental gap.

The main goal of this paper is that the characterizations of \cite{Lizorkin1987generalized,triebel1983theory} are still valid, even if we consider certain variable smoothness and the Muckenhoupt $A_p$ weights on $\bR^d$.
Precisely, we suggest a class of variable smoothness $\phi$ for given $A_p$ classes.
As an appropriate choice for $\phi$, we consider a class $\cI_o(a,b)$ with $a, b\in\bR$ introduced in \cite{Gus_Pee1977}, where we describe the class in Definition~\ref{def_pseudoconcave}.
Then, for each $A_p$ class, $a,b$ of $\cI_o(a,b)$ are given in terms of self-improving and reverse H\"older properties of $A_p$ weights.
Moreover, in the case of Besov spaces, we show that $a$ of $\cI_o(a,b)$ is a threshold for our characterizations in restricted situations (see Remark~\ref{rem:mainthm}).
That is, one can say \emph{weights determine regularity}.
Similar phenomena also occur in studies of trace theorem for Volterra type equations \cite{CLSW2023trace} and regularity estimates for parabolic equations with pseudo-differential operators \cite{CKL2023regularity}, where function spaces equipped with $\phi \in \cI_o(a,b)$ and $w\in A_p(\bR^d)$ naturally arise.

Another motivation for our study arises from recent research on initial value problems of non-local evolution equations.
For instance, H. Dong and D. Kim \cite{Dong_Kim2021} considered the following time fractional parabolic equations
\begin{equation}\label{eq:timefrac}
\partial_t^{\alpha}u(t,x)=a^{ij}(t,x)u_{x^ix^j}(t,x), \,\, (t,x)\in(0,T)\times\bR^d,
\end{equation}
and they introduced a weighted Besov space as an initial data space in the sense of \eqref{Besov_norm_diff1} whose integral is taken over $\bR^d$.
On the other hand, D. Kim and K. Woo \cite{Kim_Woo2023} suggested another weighted Besov space given in terms of \eqref{Besov_norm_lp} as an initial data space in the result of a trace theorem for equation \eqref{eq:timefrac}.
One can ask whether initial data spaces arising in \cite{Dong_Kim2021} and \cite{Kim_Woo2023} are equivalent.
This kind of question is required for a complete well-posedness result on \eqref{eq:timefrac} and we present a negative answer to this question in the first part of Section~\ref{sec_apps}.
Thus, in the sense of the trace theorem in \cite{Kim_Woo2023}, one can say that Theorem~\ref{thm_lp_diff_equiv} suggests a suitable initial data space for equations \eqref{eq:timefrac} in the sense of \eqref{Besov_norm_diff1}.

\hfill

To introduce our main result, we first define function classes $\cI$ and $\cI_o$ related to the variable smoothness.
Definitions of $A_p$ weights and the reverse H\"older property are also given.
Then, we define the function spaces concerned in this paper.

\begin{defn}\label{def_pseudoconcave}
For a function $\phi:\bR_+ \left( := \left(0, \infty \right) \right)  \rightarrow \bR_+$, we define 
$$
s_\phi(\lambda) := \sup_{t>0}\frac{\phi(\lambda t)}{\phi(t)}\,,
$$
so that $s_{\phi}:\bR_+\rightarrow (0,\infty]$.
Observe that $s_\phi$ is \textit{submultiplicative}, \textit{i.e}., $s_\phi(\lambda\tau) \leq s_\phi(\lambda) s_\phi(\tau)$ for $\lambda, \tau>0$.
Together with the definition of $s_\phi(\lambda)$, we introduce the following class of functions.
For $a, b \in \mathbb{R}$ with $a \le b$, we set
	\begin{align}\label{ftn class}
		\mathcal{I}(a, b) := \{\phi : \textrm{$s_\phi(\lambda) = O(\lambda^a)$ as $\lambda \to 0$, and $s_\phi(\lambda) = O(\lambda^b)$ as $\lambda \to \infty$} \}.
	\end{align}
We also define $\mathcal{I}_o(a,b)$, where $o$ denotes that we replace $O$-notation with $o$-notation in \eqref{ftn class}. Clearly, $\mathcal{I}_o(a,b)\subsetneq \mathcal{I}(a,b)$.
\end{defn}

Note that for any $\phi \in \mathcal{I}(a,b) \neq \emptyset$ (or $\phi \in \mathcal{I}_o(a,b) \neq \emptyset$), one necessarily has $a \le b$.
Indeed, since $s\phi(\lambda)s_\phi(1/\lambda) \ge 1$ for all $\lambda>0$, the asymptotic bounds yield, for sufficiently large $\lambda$,
$$
1\leq s_{\phi}(\lambda)s_{\phi}(1/\lambda)\lesssim \lambda^{b}\lambda^{-a}=\lambda^{b-a},
$$
and hence $a\leq b$.
Throughout the paper, $d \in \mathbb{N}$ denotes the spatial dimension of $\mathbb{R}^d$.

\begin{defn}
	\begin{enumerate}[(i)]
\item     For $p\in(1,\infty)$, let $A_p=A_p(\bR^d)$ be the class of all nonnegative and locally integrable functions $w$ satisfying
\begin{align}
\label{def ap}
[w]_{A_p}:=\sup_{\cQ\,\text{cubes in }\bR^d}\left(\aint_{\cQ}w(x)dx\right)\left(\aint_{\cQ}w(x)^{-1/(p-1)}dx\right)^{p-1}<\infty.
\end{align}
The class of $A_1=A_1(\bR^d)$ is defined as a collection of nonnegative and locally integrable functions $w$ satisfying
$$
\cM w(x)\leq Cw(x) \quad \forall x\in\bR^d,
$$
where $\cM$ is the Hardy-Littlewood maximal operator and $C$ is independent of $x$.
The class $A_{\infty}=A_{\infty}(\bR^d)$ could be defined as the union of $A_p(\bR^d)$ for all $p\in[1,\infty)$, \textit{i.e.},
    $$
    A_\infty(\bR^d)=\bigcup_{p\in[1,\infty)}A_p(\bR^d).
    $$
\item For $s\in(1,\infty)$, let $RH_s=RH_s(\bR^d)$ be the class of all nonnegative and locally integrable functions $w$ satisfying
$$
[w]_{RH_{s}}:=\sup_{\cQ\,\text{cubes in }\bR^d}\left(\aint_{\cQ}w(x)^s\mathrm{d}x\right)^{1/s}\left(\aint_{\cQ}w(x)\mathrm{d}x\right)^{-1} < \infty.
$$
The class of $RH_{\infty}=RH_{\infty}(\bR^d)$ is defined as a collection of nonnegative and locally integrable functions $w$ satisfying
$$
\esssup_{x\in \cQ}w(x)\leq C\aint_{\cQ}w(x)\mathrm{d}x
$$
for all cubes $\cQ\subset \bR^d$. Here, $C$ is independent of $\cQ$.
\end{enumerate}
\end{defn}

$\mathcal{S}(\mathbb{R}^d)$ and $\mathcal{S}'(\mathbb{R}^d)$ denote the Schwartz class and its dual, respectively.
Let a function $\varphi$ be of $\mathcal{S}(\mathbb{R}^d)$ whose Fourier transform $\widehat{\varphi}$ is nonnegative, supported in a ball $B_2(0)$, and let $\widehat{\psi}(\cdot) = \widehat{\varphi}(\cdot) - \widehat{\varphi}(2 \cdot)$ so that $\sum_{j\in\mathbb{Z}} \widehat{\psi}_j = 1$ for $\xi \not=0$ where $\psi_j(\cdot) = 2^{jd} \psi(2^j\cdot)$.
Then for $f\in\cS'(\bR^d)$, we define the Littlewood-Paley projection operators $\Delta_j$ and $S_0$ by 
\begin{align}\label{ineq_230315_1246}
\widehat{\Delta_j f}(\xi) := \widehat{\psi}_j(\xi) \widehat{f}(\xi),\quad S_0f := \sum_{j\leq 0} \Delta_j f = \Phi \ast f,
\end{align}
respectively.
Here, $\Phi$ denotes a smooth low–frequency cutoff function defined by 
$\Phi := \sum_{j \leq 0} \psi_j$. 
Since $\widehat{\Phi}$ is an infinitely differentiable function with compact support $\overline{B_2(0)}$, it follows that $\Phi \in \mathcal{S}(\mathbb{R}^d)$ as well.
Together with \eqref{ineq_230315_1246}, we define weighted Besov and Triebel-Lizorkin spaces with variable smoothness.
\begin{defn}
\label{def}
    Let $p\in(1,\infty]$, $q\in[1,\infty]$, $M>0$ and $\phi\in \cI_o(0,M)$. 
If $p<\infty$, then we set $w\in A_p(\bR^d)$ and if $p=\infty$, then we set $L_p(\bR^d,w\,\mathrm{d}x):=L_{\infty}(\bR^d)$. 
    \begin{enumerate}[(i)]
    \item
    The space $B_{p,q}^{\phi}(\bR^d,w\,\mathrm{d}x)$ is defined by
    $$
    B_{p,q}^{\phi}(\bR^d,w\,\mathrm{d}x):=\{f\in\cS'(\bR^d):\|f\|_{B_{p,q}^{\phi}(\bR^d,w\,\mathrm{d}x)}<\infty\},
    $$
    where the norm $\|f\|_{B_{p,q}^{\phi}(\bR^d,w\,\mathrm{d}x)}$ is defined by
    \begin{eqnarray*}
    \|S_0f\|_{L_p(\bR^d,w\,\mathrm{d}x)}
    +\left\{
    \begin{array}{ll}
    \Big(\sum_{j=1}^{\infty}\phi(2^{j})^q\|\Delta_jf\|_{L_p(\bR^d,w\,\mathrm{d}x)}^q\Big)^{1/q}, &q<\infty,\\
    \sup_{j\in\bN}\phi(2^j)\|\Delta_jf\|_{L_p(\bR^d,w\,\mathrm{d}x)}, &q=\infty.
    \end{array}
    \right.
    \end{eqnarray*}
    We also define the homogeneous Besov norm
    \begin{eqnarray*}
    \|f\|_{\mathring{B}_{p,q}^{\phi}(\bR^d,w\,\mathrm{d}x)}
    :=\left\{
    \begin{array}{ll}
    \left(\sum_{j\in\bZ}\phi(2^j)^q\|\Delta_jf\|_{L_p(\bR,w\,\mathrm{d}x)}^q\right)^{1/q}, &q<\infty,\\
    \sup_{j\in\bZ}\phi(2^j)\|\Delta_jf\|_{L_p(\bR^d,w\,\mathrm{d}x)}, &q=\infty.
    \end{array}
    \right.
    \end{eqnarray*}
    
    \item
    The space $F_{p,q}^{\phi}(\bR^d,w\,\mathrm{d}x)$ is defined by
    $$
    F_{p,q}^{\phi}(\bR^d,w\,\mathrm{d}x):=\{f\in\cS'(\bR^d):\|f\|_{F_{p,q}^{\phi}(\bR^d,w\,\mathrm{d}x)}<\infty\},
    $$
    where the norm $\|f\|_{F_{p,q}^{\phi}(\bR^d,w\,\mathrm{d}x)}$ is defined by
    \begin{eqnarray*}
    \|S_0f\|_{L_p(\bR^d,w\,\mathrm{d}x)}
    + \left\{
    \begin{array}{ll}
    \Big\|\Big(\sum_{j=1}^{\infty}\phi(2^{j})^q|\Delta_jf|^q\Big)^{1/q}\Big\|_{L_p(\bR^d,w\,\mathrm{d}x)}, &q<\infty,\\
    \Big\|\sup_{j\in\bN}\phi(2^j)|\Delta_jf|\Big\|_{L_p(\bR^d,w\,\mathrm{d}x)}, &q=\infty.
    \end{array}
     \right.
    \end{eqnarray*}
    We also define the homogeneous Triebel-Lizorkin norm
    \begin{eqnarray*}
    \|f\|_{\mathring{F}_{p,q}^{\phi}(\bR^d,w\,\mathrm{d}x)}
    :=\left\{
    \begin{array}{ll}
    \left\|\left(\sum_{j\in\bZ}\phi(2^j)^q|\Delta_jf|^q\right)^{1/q}\right\|_{L_p(\bR,w\,\mathrm{d}x)}, &q<\infty,\\
    \left\|\sup_{j\in\bZ}\phi(2^j)|\Delta_jf|\right\|_{L_p(\bR^d,w\,\mathrm{d}x)}, &q=\infty.
    \end{array}
    \right.
    \end{eqnarray*}
    \end{enumerate}
\end{defn}
It is worth noting that the norms of 
$B_{p,q}^{\phi}(\mathbb{R}^d, w\,\mathrm{d}x)$ and $F_{p,q}^{\phi}(\mathbb{R}^d, w\,\mathrm{d}x)$ are independent of the choice of $\{\psi_j\}_{j\ge0}$. 
For a detailed justification of this independence, we refer to the discussion following Theorem 2.4 in \cite{Bui1982} and the references therein.

For the norms given in terms of difference operators, we define

\begin{defn}\label{def_diff}
	Let $p\in(1,\infty]$, $q\in[1,\infty]$, $M>0$, $\phi\in \cI_o(0,M)$, and denote by $L$ the smallest integer not less than $M$.
If $p<\infty$, then we set $w\in A_p(\bR^d)$ and if $p=\infty$, then we set $L_p(\bR^d,w\,\mathrm{d}x):=L_{\infty}(\bR^d)$. 
For a given function $f$, let $\cD_h^1 f (x) = \cD_h f (x) = f(x+h) - f(x)$ and $\cD_h^L = \cD_h \cD_h^{L-1}$ for $L = 2, 3, \cdots$.

    \begin{enumerate}[(i)]
    \item
    The space $\cB_{p,q}^{\phi}(\bR^d,w\,\mathrm{d}x)$ is defined by
    $$
    \cB_{p,q}^{\phi}(\bR^d,w\,\mathrm{d}x):=\{f\in\cS'(\bR^d):\|f\|_{\cB_{p,q}^{\phi}(\bR^d,w\,\mathrm{d}x)}<\infty\},
    $$
    where $\|f\|_{\cB_{p,q}^{\phi}(\bR^d,w\,\mathrm{d}x)}$ is given by
    \begin{eqnarray}\label{wBesov_norm_diff}
    \|f\|_{\cB_{p,q}^{\phi}(\bR^d,w\,\mathrm{d}x)}:=\|f\|_{L_p(\bR^d,w\,\mathrm{d}x)}
    +\|f\|_{\mathring{\cB}_{p,q}^\phi(\bR^d,w\dd x)},
    \end{eqnarray}
    and $\|f\|_{\mathring{\cB}_{p,q}^{\phi}(\bR^d,w\,\mathrm{d}x)}$ is also given by
    $$
    \|f\|_{\mathring{\cB}_{p,q}^\phi(\bR^d,w\dd x)}:=\left\{\begin{array}{ll} \Big(\int_{|h|\leq 1} \phi(|h|^{-1})^q  \left\| \cD_h^L f \right\|_{L_p(\bR^d, w\,\mathrm{d}x)}^q~\frac{\mathrm{d}h}{|h|^d}\Big)^{1/q}, &q<\infty,\\
    \sup_{|h|\leq 1}\phi(|h|^{-1}) \left\| \cD_h^L f \right\|_{L_p(\bR^d, w\,\mathrm{d}x)},&q=\infty.
    \end{array}
    \right.
    $$
    We also define the space $\mathscr{B}_{p,q}^\phi(\bR^d, w\,\mathrm{d}x)$ together with a norm given by \eqref{wBesov_norm_diff} with $\bR^d$ instead of $|h|\leq 1$.
    \item
    The space $\cF_{p,q}^{\phi}(\bR^d,w\,\mathrm{d}x)$ is defined by
    $$
    \cF_{p,q}^{\phi}(\bR^d,w\,\mathrm{d}x):=\{f\in\cS'(\bR^d):\|f\|_{\cF_{p,q}^{\phi}(\bR^d,w\,\mathrm{d}x)}<\infty\},
    $$
    where $\|f\|_{\cF_{p,q}^{\phi}(\bR^d,w\,\mathrm{d}x)}$ is given by
    \begin{eqnarray}\label{wTL_norm_diff}
    \|f\|_{\cF_{p,q}^{\phi}(\bR^d,w\,\mathrm{d}x)}:=\|f\|_{L_p(\bR^d,w\,\mathrm{d}x)}
    +\|f\|_{\mathring{\cF}_{p,q}^\phi(\bR^d,w\dd x)},
    \end{eqnarray}
    and $\|f\|_{\mathring{\cF}_{p,q}^{\phi}(\bR^d,w\,\mathrm{d}x)}$ is also given by
    $$
    \|f\|_{\mathring{\cF}_{p,q}^{\phi}(\bR^d,w\,\mathrm{d}x)}:=\left\{\begin{array}{ll}\Big\| \Big(\int_{|h|\leq 1} \phi(|h|^{-1})^q   |\cD_h^L f|^q~\frac{\mathrm{d}h}{|h|^d}\Big)^{1/q} \Big\|_{L_p(\bR^d, w\,\mathrm{d}x)}, &q<\infty,\\
    \Big\| \sup_{|h|\leq 1}\phi(|h|^{-1}) |\cD_h^L f| \Big\|_{L_p(\bR^d, w\,\mathrm{d}x)},&q=\infty.
    \end{array}
    \right.
    $$
    Similarly, we define the space $\mathscr{F}_{p,q}^\phi(\bR^d, w\,\mathrm{d}x)$ together with the norm given by \eqref{wTL_norm_diff} with $\bR^d$ instead of $|h|\leq 1$.
    \end{enumerate}
\end{defn}
It can be easily checked that all spaces introduced in Definitions \ref{def} and \ref{def_diff} are Banach spaces.
Together with Definitions~\ref{def} and \ref{def_diff}, we state our main results. 
\begin{thm}\label{thm_lp_diff_equiv}
Let $w \in A_p(\mathbb{R}^d)$ with $p \in (1, \infty]$, and we assume that
$$
M > \frac{d}{p}\left(R_w + \frac{1}{\Gamma_w} - 1\right).
$$ 
	Then 
	\begin{enumerate}[(i)]
	\item for $p\in(1,\infty]$, $q\in[1,\infty]$ and $\phi\in\cI_o \left(\frac{d}{p}\left(R_w+\frac{1}{\Gamma_w}-1\right),M \right)$,
 $$
 B_{p,q}^\phi(\bR^d, w~\mathrm{d}x)=\cB_{p,q}^\phi(\bR^d, w~\mathrm{d}x)=\mathscr{B}_{p,q}^\phi(\bR^d, w\,\mathrm{d}x).
 $$
	\item for $p,q\in(1, \infty)$ and $\phi\in\cI_o(\frac{dR_{w}}{p},M)$,
 $$
 F_{p,q}^\phi(\bR^d, w~\mathrm{d}x)=\cF_{p,q}^\phi(\bR^d, w~\mathrm{d}x)=\mathscr{F}_{p,q}^\phi(\bR^d, w\,\mathrm{d}x).
 $$
 \end{enumerate}
 Here, the constants $R_w$ and $\Gamma_w$ are defined by
 $$
 R_w:=\inf\{q:w\in A_q(\bR^d)\}\in[1,p],\quad \Gamma_w:=\sup\{s:w\in RH_s(\bR^d)\}\in[1,\infty].
 $$
\end{thm}
The proof of Theorem~\ref{thm_lp_diff_equiv} follows from several inequalities that depend on all the parameters in Definition~\ref{def_diff}. 
The dependence on these parameters will appear only implicitly, as they are fixed throughout the argument and have no effect on the statement of Theorem~\ref{thm_lp_diff_equiv}.

Note that the constants $R_w$, $\Gamma_w$ have frequently appeared in the literature \cite{Bui1982,Bui1984,BuiDuong_2020,BPT1996}.
We also obtain similar results for homogeneous norms under the same conditions of Theorem~\ref{thm_lp_diff_equiv}.
\begin{thm} 
\label{homo}
If all of conditions in Theorem \ref{thm_lp_diff_equiv} hold, then for $f\in X_{p,q}^{\phi}(\bR^d,w\,\mathrm{d}x)$,
		\begin{align*}
			\|f\|_{\mathring{X}_{p,q}^\phi(\bR^d,w\dd x)}\simeq \|f\|_{\mathring{\mathscr{X}}_{p,q}^{\phi}(\bR^d,w\,\mathrm{d}x)}.
   \end{align*}
Here, $(X,\mathscr{X})=(B,\mathscr{B})$, $(F,\mathscr{F})$. 
\end{thm}
The proofs of Theorems \ref{thm_lp_diff_equiv} and \ref{homo} are provided in Sections \ref{sec_mainproof} and \ref{sec_prop_equiv}, respectively.
The cases $q=\infty$ for $F$, $\mathcal{F}$ and $\mathscr{F}$ are excluded in Theorem \ref{thm_lp_diff_equiv} because they do not hold even for unweighted cases.
The proof of Theorem \ref{thm_lp_diff_equiv} is based on Proposition \ref{prop_wBesov_equiv},
which consists of various quantities equivalent to $\|f\|_{B_{p,q}^\phi(\bR^d, w\,\mathrm{d}x)}$ and $\|f\|_{F_{p,q}^\phi(\bR^d, w\,\mathrm{d}x)}$.
Assuming Proposition~\ref{prop_wBesov_equiv}, we will show in Lemma~\ref{23.02.08.15.58} that 
$$
\|f\|_{X_{p,q}^\phi(\bR^d, w\,\mathrm{d}x)} 
\lesssim \|f\|_{\cX_{p,q}^\phi(\bR^d, w\,\mathrm{d}x)},
$$ 
where $(X, \cX, \mathscr{X}) = (B, \cB, \mathscr{B})$ or $(F, \cF, \mathscr{F})$.
We also show in Lemma~\ref{23.02.09.14.46} that
$$
\|f\|_{\mathscr{X}_{p,q}^\phi(\bR^d, w\,\mathrm{d}x)}
\lesssim \|f\|_{X_{p,q}^\phi(\bR^d,w\,\mathrm{d}x)}.
$$ 
In proving Lemma~\ref{23.02.09.14.46}, we make use of shifted maximal functions and Peetre's maximal functions for $(X,\mathscr{X})=(B,\mathscr{B})$ and $(X,\mathscr{X})=(F,\mathscr{F})$, respectively.
Detailed properties of these maximal functions are presented in Section~\ref{sec_preliminaries}.

Now we give some remarks on our assumption $\phi$ in Theorem \ref{thm_lp_diff_equiv}.
\begin{rem}\label{rem:mainthm}
$(i)$ If $w\equiv1$, then
\begin{align}\label{230612535}
R_{w}=1,\quad \Gamma_w=\infty.
\end{align}
Thus,
$$
\frac{d}{p}\left(R_w+\frac{1}{\Gamma_w}-1\right)=0\quad \text{and}\quad \frac{dR_w}{p}=\frac{d}{p}\,.
$$
Hence Theorem~\ref{thm_lp_diff_equiv}-$(i)$ contains non-weighted characterizations introduced in \cite{triebel1983theory}.
Similarly, it can be easily checked that Theorem~\ref{thm_lp_diff_equiv}-$(ii)$ also contains non-weighted characterizations introduced in \cite{triebel1983theory}.

It is worth noting that any weights satisfying \eqref{230612535} suggest the largest class of $\phi$ for which Theorem~\ref{thm_lp_diff_equiv} holds.
One of sufficient conditions for \eqref{230612535} is  $w\in A_1\cap RH_\infty$, which means that
\begin{equation}
\label{25.11.04.13.10}
	w(x) \simeq \cM w(x)\simeq \inf_{ Q\ni x}\aint_{\cQ}w(y)\,\mathrm{d}y, \quad \forall x\in\bR^d.
\end{equation}
Moreover, it is proved in \cite[Theorem 5.1]{Kin_Sh2014} that \eqref{25.11.04.13.10} holds if and only if $\log w\in \overline{L^{\infty}(\bR^d)}\subseteq \mathrm{BMO}(\bR^d)$.

$(ii)$ If $\phi \not\in \cI_o \left(\frac{d}{p}\left(R_w+\frac{1}{\Gamma_w}-1\right),M \right)$, then there is a counter example for Theorem~\ref{thm_lp_diff_equiv}-$(i)$.
That is, there is a certain weight $w$ such that if $\phi\in \cI_o \left( 0,\frac{d}{p}\left(R_w+\frac{1}{\Gamma_w}-1\right) \right)$, then 
\begin{align*}
    \mathscr{B}_{p,q}^{\phi}(\bR^d,w\,\mathrm{d}x)\subsetneq \cB_{p,q}^{\phi}(\bR^d,w\,\mathrm{d}x)\subsetneq B_{p,q}^{\phi}(\bR^d,w\,\mathrm{d}x).
\end{align*}
Indeed, let $d=1$, $p,q\in(1,\infty)$, $\alpha\in(0,p-1)$, and $w(x)=|x|^{\alpha}$. Then, it follows that $w$ is in $A_p(\bR)$, $R_{w}=1+\alpha$, and $1/\Gamma_w=0$.
Consequently,
$$
\frac{d}{p}\left(R_{w}+\frac{1}{\Gamma_w}-1\right)=\frac{\alpha}{p}.
$$
Assume that $\phi(\lambda)=\lambda^{\theta}$, where $\theta\leq \alpha/p$, and $f_0 \in C_c^{\infty}(\bR)$, which satisfies two conditions:
\begin{itemize}
\item $|f_0(x)|\geq1$ for $x\in(-1,1)$,
\item $f_0(x)=0$ for $x\geq2$ or $x\leq-2$.
\end{itemize}
We then have that $f_0$ is in the space $\cB_{p,q}^{\phi}(\bR,w\,\mathrm{d}x)$ and $\phi$ belongs to $\cI_o(0,\alpha/p)$.
Using Proposition \ref{prop_wBesov_equiv} and Lemma \ref{23.02.08.15.58}, we can deduce that
$$
\cB_{p,q}^{\phi}(\bR^d,w\,\mathrm{d}x)\subseteq B_{p,q}^{\phi}(\bR^d,w\,\mathrm{d}x),
$$
thus, $f_0\in B_{p,q}^{\phi}(\bR^d,w\,\mathrm{d}x)$.

First, we verify that
\begin{equation}
	\label{eq230928_01}
\mathscr{B}_{p,q}^{\phi}(\bR,w\,\mathrm{d}x)\subsetneq \cB_{p,q}^{\phi}(\bR,w\,\mathrm{d}x).
\end{equation}
Since for $h>4$,
$$
\|f_0(\cdot+h)-f_0\|_{L_p(\bR,w\,\mathrm{d}x)}^p\geq\int_{-\infty}^{\infty}|f_0(x+h)|^p|x|^{\alpha}\mathrm{d}x\geq \int_{-1}^{1}|x-h|^{\alpha}\mathrm{d}x\geq Nh^{\alpha}
$$
we have
$$
\int_{\bR}\frac{\|f_0(\cdot+h)-f_0\|_{L_p(\bR,w\,\mathrm{d}x)}^q}{|h|^{1+q\theta}}\mathrm{d}h\geq \int_{4}^{\infty}h^{q(\frac{\alpha}{p}-\theta)-1}\mathrm{d}h=+\infty,
$$
and this implies \eqref{eq230928_01}.

To demonstrate that
$$
\cB_{p,q}^{\phi}(\bR^d,w\,\mathrm{d}x)\subsetneq B_{p,q}^{\phi}(\bR^d,w\,\mathrm{d}x),
$$
we need to prove that the inclusion mapping $I:\cB_{p,q}^{\phi}(\bR^d,w\mathrm{d}x)\to B_{p,q}^{\phi}(\bR^d,w\mathrm{d}x)$ is not surjective.
To establish this, we assume, for the sake of contradiction, that $I$ is surjective.
Then, by Proposition \ref{prop_wBesov_equiv} and Lemma \ref{23.02.08.15.58}, $I$ is bounded linear and bijective.
Hence \cite[Theorem 5.10]{rudin2006real}, we have
\begin{align}
\label{23.06.14.16.17}
    \|g\|_{\cB_{p,q}^{\phi}(\bR^d,w\,\mathrm{d}x)}\lesssim \|I(g)\|_{B_{p,q}^{\phi}(\bR^d,w\,\mathrm{d}x)}=\|g\|_{B_{p,q}^{\phi}(\bR^d,w\,\mathrm{d}x)}, \,\, \forall g\in \cB_{p,q}^{\phi}(\bR^d,w\,\mathrm{d}x).
\end{align}
However, under the assumption of Theorem \ref{thm_lp_diff_equiv}-$(i)$, \eqref{23.06.14.16.17} fails.
To verify this, put $f_n(x):=2^{n/p}f_0(2^nx)$.
Then for $n\geq3$,
\begin{align}
\label{23.06.14.18.51}
    \int_{|h|\leq1}\frac{\|f_n(\cdot+h)-f_n\|_{L_p(\bR,w\,\mathrm{d}x)}^q}{|h|^{1+q\theta}}\mathrm{d}h= 2^{qn\left(\theta-\alpha/p\right)}\int_{|h|\leq2^n}\frac{\|f_0(\cdot+h)-f_0\|_{L_p(\bR,w\,\mathrm{d}x)}^q}{|h|^{1+q\theta}}\mathrm{d}h.
\end{align}
Moreover,
\begin{equation}
\label{23.06.14.18.52}
    \begin{gathered}
        \|f_n\|_{L_p(\bR,w\,\mathrm{d}x)}=2^{-n\alpha/p}\|f_0\|_{L_p(\bR,w\,\mathrm{d}x)},\\
        \|f_n\|_{\mathring{B}_{p,q}^{\phi}(\bR,w\,\mathrm{d}x)}:=\left(\sum_{j\in\bZ}\phi(2^j)^q\|\Delta_jf_n\|_{L_p(\bR,w\,\mathrm{d}x)}^q\right)^{1/q} =2^{n\left(\theta-\alpha/p\right)}\|f_0\|_{\mathring{B}_{p,q}^{\phi}(\bR,w\,\mathrm{d}x)}.
    \end{gathered}
\end{equation}
Therefore, by \eqref{23.06.14.16.17}, \eqref{23.06.14.18.51} and \eqref{23.06.14.18.52}
\begin{align*}
    &2^{-n\alpha/p}\|f_0\|_{L_p(\bR,w\,\mathrm{d}x)}+2^{n\left(\theta-\alpha/p\right)}\left(\int_{4}^{2^n}|h|^{q(\alpha/p-\theta)-1}\mathrm{d}h\right)^{1/q}\\
    &\leq 2^{-\frac{n\alpha}{p}}\|f_0\|_{L_p(\bR,w\,\mathrm{d}x)}+2^{n\left(\theta-\alpha/p\right)}\left(\int_{|h|\leq2^n}\frac{\|f_0(\cdot+h)-f_0\|_{L_p(\bR,w\,\mathrm{d}x)}^q}{|h|^{1+q\theta}}\mathrm{d}h\right)^{1/q}\\
    &=\|f_n\|_{\cB_{p,q}^{\phi}(\bR^d,w\,\mathrm{d}x)}\\
    &\lesssim \|f_n\|_{B_{p,q}^{\phi}(\bR^d,w\,\mathrm{d}x)}\simeq 2^{-n\alpha/p}\|f_0\|_{L_p(\bR,w\,\mathrm{d}x)}+2^{n\left(\theta-\alpha/p\right)}\|f_0\|_{\mathring{B}_{p,q}^{\phi}(\bR^d,w\,\mathrm{d}x)}.
\end{align*}
The last equivalence follows from Proposition \ref{prop_wBesov_equiv}.
This certainly implies
$$
+\infty=\lim_{n\to\infty}\left(\int_{4}^{2^n}|h|^{q(\alpha/p-\theta)-1}\mathrm{d}h\right)^{1/q}\lesssim\|f_0\|_{\mathring{B}_{p,q}^{\phi}(\bR^d,w\,\mathrm{d}x)},
$$
which contradicts the fact that $f_0$ is in $\mathring{B}_{p,q}^{\phi}(\bR^d,w\,\mathrm{d}x)$.
\end{rem}

For mathematical completeness in function space theory, one might be curious about some analogous version of the equivalence results in Theorem~\ref{thm_lp_diff_equiv}.
For example,
\begin{itemize}
    \item characterizations of $X^{\phi}_{p,q}(\bR^d,w\,\mathrm{d}x)$ for $p,q \in (0,1]$.
    \item characterizations of homogeneous function spaces $\mathring{X}^{\phi}_{p,q}(\bR^d,w\,\mathrm{d}x)$.
\end{itemize}
We first comment on the result for $p,q \in (0,1]$.
To accomplish such a goal, we need a \emph{Hardy space theory}, which guarantees continuity of norms in the sense of $p,q \in (0,\infty)$ and negative regularity.
H.-Q. Bui, T. A. Bui, and X. T. Duong \cite{BuiDuong_2020} make use of results in Hardy space to characterize Besov and Triebel-Lizorkin spaces in terms of continuous norms and square function type norms.
Those full characterizations, however, are not in our main concern, which contains trace and Sobolev embedding theorems.
Note that function spaces arising in studies of trace and Sobolev embedding results have indices between $1$ and $\infty$, so we don't discuss characterizations for $p,q\in(0,1]$ in this paper.
For readers interested in full characterizations for $p,q \in(0,\infty)$, we recommend \cite{BuiDuong_2020} and references therein.

For the characterizations of homogeneous spaces, one may desire to replace $X_{p,q}^\phi(\bR^d,w\,\mathrm{d}x)$ into $\mathring{X}^{\phi}_{p,q}(\bR^d,w\,\mathrm{d}x)$ in Theorem~\ref{homo}.
We end this section with comments on the spaces $\mathring{X}^{\phi}_{p,q}(\bR^d,w\,\mathrm{d}x)$.
For $\mathring{X}^{\phi}_{p,q}(\bR^d,w\,\mathrm{d}x)$ to be a Banach space with the norm $\|\cdot\|_{\mathring{X}^{\phi}_{p,q}(\bR^d,w\,\mathrm{d}x)}$, this space should be defined in $\cS'(\bR^d)/\cP(\bR^d)$ according to classical results regarding homogeneous function spaces (see \textit{e.g.} \cite{sawano2019}).
Here, $\cS'(\bR^d)/\cP(\bR^d)$ is the quotient space of $\cS'(\bR^d)$ by the polynomial ring, $\cP(\bR^d)$.
As all Fourier transforms of polynomials are supported on $\{0\}$, $\|\cdot\|_{\mathring{X}^{\phi}_{p,q}(\bR^d,w\,\mathrm{d}x)}$ is well-defined in $\cS'(\bR^d)/\cP(\bR^d)$.
We propose that $\|\cdot \|_{\mathring{\mathscr{X}}_{p,q}^\phi(\bR^d)}$ is well-defined under the modulo of polynomials due to the following proposition:
\begin{prop}
\label{poly}
	Let $p,\,q\in (1,\infty)$, {\color{blue}$L>0$} and $\phi\in\cI_o(0,L)$.
 For any polynomial $P$,
	\begin{align}\label{230714638}	\|P\|_{\mathring{\mathscr{X}}_{p,q}^\phi(\bR^d)}=0\quad\text{if}\,\,\,\,\deg P<L\quad\text{and}\quad \|P\|_{\mathring{\mathscr{X}}_{p,q}^\phi(\bR^d)}=\infty\quad\text{if}\,\,\,\,\deg P\geq L.
	\end{align}
\end{prop}
Characterizations of $\mathring{X}_{p,q}^{\phi}(\bR^d,w\,\mathrm{d}x)\subset \cS'(\bR^d)/\cP(\bR^d)$ may be possible together with Proposition~\ref{poly}, but we do not pursue further this as it falls beyond the scope of this paper and would require a separate discussion on the realization of $\mathring{X}^{\phi}_{p,q}(\bR^d,w\,\mathrm{d}x)$ to $L_{1,\mathrm{loc}}(\bR^d)$.
The proof of Proposition \ref{poly} will be provided in the last part of Section \ref{sec_prop_equiv}.

\mysection{Preliminaries}\label{sec_preliminaries}

In this section, we introduce properties of $A_p$ weights, functions of type $\cI_o(a,b)$, shifted maximal functions, and Peetre's maximal functions.
Some weighted vector-valued inequalities are also given.
Finally, we suggest various quantities equivalent with $B_{p,q}^\phi(\bR^d, w~\mathrm{d}x)$ and $F_{p,q}^\phi(\bR^d, w~\mathrm{d}x)$.

\subsection{Self-improving and the reverse H\"older properties of the Muckenthoupt $A_p$ class}
\label{23.05.27.13.01}

It is well known that for $w\in A_p$ we have
	\begin{align}\label{2306111001}
	\|\cM\|_{L_p(w)\rightarrow L_p(w)}\lesssim_{d,p}[w]_{A_p}^{1/(p-1)}.
\end{align}
In addition, we introduce useful properties of $w\in A_p(\bR^d)$ used in this paper.

\begin{prop}
\label{23.04.21.16.12}
Let $p,s\in(1,\infty)$, $\cQ$ be a cube and $S\subset \cQ$. 
\begin{enumerate}[(i)]
    \item If $w\in A_p(\bR^d)$, then $w\in A_q(\bR^d)$ for some $q<p$.
    \item If $w\in A_p(\bR^d)$, then 
    $$
    \left(\frac{|S|}{|\mathcal{Q}|}\right)^p \leq [w]_{A_p}\frac{w(S)}{w(\mathcal{Q})}\,.
    $$
    \item If $w\in A_p(\bR^d)$, then there exists a constant $s=s(d,p,[w]_{A_p})>1$ such that $w\in RH_{s}(\bR^d)$
    \item If $w\in RH_s(\bR^d)$, then $w\in RH_{t}(\bR^d)$ for some $t>s$.
    \item If $w\in RH_s(\bR^d)$, then
    $$
    \frac{w(S)}{w(\mathcal{Q})}\leq [w]_{RH_s}\left(\frac{|S|}{|\mathcal{Q}|}\right)^{1-1/s} \,.
    $$
\end{enumerate}
\end{prop}
\begin{proof}
The first three assertions are in \cite[Corollary 7.2.6, Lemma 9.2.1, Theorem 7.2.2]{grafakos2014classical}, respectively. 
The fourth assertion can be found in \cite[Lemma 3]{G1973}.
Thus, we only prove the last assertion.
Let $w$ be of class $RH_s(\bR^d)$.
Then 
\begin{align}\label{ineq_RH}
	\Big(\frac{1}{|\cQ|}\int_{\cQ} w(x)^s~\mathrm{d}x\Big)^{1/s} \leq [w]_{RH_s} \frac{1}{|\cQ|}\int_\cQ w(x)~\mathrm{d}x.
\end{align}
By making use of \eqref{ineq_RH}, it follows that for a subset $S$ of a cube $\cQ$
\begin{align*}
    w(S) 
    = \int_\cQ \mathbbm{1}_S(x) w(x)~\mathrm{d}x
    \leq \Big(\int_\cQ w(x)^s~\mathrm{d}x\Big)^{\frac{1}{s}} |S|^{1-1/s}
    \leq [w]_{RH_s} w(\cQ) \left(\frac{|S|}{|\cQ|}\right)^{1-1/s}.
\end{align*}
\end{proof}

With the help of Proposition \ref{23.04.21.16.12}, the constants $R_w$ and $\Gamma_w$ are well-defined for $w\in A_p(\bR^d)$.
We introduce a new constant $R_{w,p}$, which is related to $R_w$ and naturally arises in estimates of Peetre's maximal functions.
\begin{align*}
R_{w,p}:=\sup\{p_0:w\in A_{p/p_0}(\bR^d)\}\in (1,p].
\end{align*}
It can be easily verified that for $p\in[1,\infty)$,
\begin{equation}
\label{23.05.29.14.53}
R_w=\frac{p}{R_{w,p}}.
\end{equation}
To maintain generality, we assume without loss of generality that $R_{w,p}=\infty$ and $p/R_{w,p}=1$ if $p=\infty$.
From this point forward, we will use $R_{w,p}$ instead of $R_w$ throughout the remainder of this paper.

\subsection{Properties of functions of type $\cI(a, b)$}

For functions given in Definition~\ref{def_pseudoconcave}, we introduce the following properties which are frequently used in our argument. For detailed proof, see \cite[Lemma 2.3]{CLSW2023trace}.
\begin{prop}
\label{23.05.25.13.12}
Let $a,b\in\bR$ with $a < b$ and $\phi\in \mathcal{I}(a,b)$.
    \begin{enumerate}[(i)]
        \item For $\alpha\in\bR$, $t^\alpha \phi(t) \in \mathcal{I}(a+\alpha, b+\alpha)$.
        \item For $\alpha\geq0$, $\phi(t^\alpha), \phi(t)^\alpha \in \mathcal{I}(\alpha a, \alpha b)$.
        \item For $\alpha \leq 0$, $\phi(t^\alpha), \phi(t)^\alpha \in \mathcal{I}(\alpha b, \alpha a)$.
        \item For $a<c$, $b>d$, $\mathcal{I}(c,d)\subset \mathcal{I}_{o}(a,b)$ and
        \begin{align*}
		\bigcup_{a< c,\,b>d} \mathcal{I}(c, d) = \mathcal{I}_o(a, b).
        \end{align*}
        \item For $\phi\in\mathcal{I}_{o}(a,b)$, there exists $\varepsilon = \varepsilon(\phi) >0$ such that
        \begin{align}\label{w scaling}
		\lambda^{a+\varepsilon}\lesssim \frac{\phi(\lambda t)}{\phi(t)} \lesssim \lambda^{b - \varepsilon},
	    \end{align}
	     for $t \in \bR_+$ and $\lambda \geq 1$.
        \item For $\phi \in \mathcal{I}_o(a, b)$, the assertions (i)--(iii) hold for $\mathcal{I}_{o}(a,b)$ instead of $ \mathcal{I}(a,b)$.

    \end{enumerate}

\end{prop}

\subsection{Shifted maximal functions}

For $z\in\bR^d$ and a dyadic cube $Q\subset\bR^d$, we define a \emph{shifted maximal function} $\cM^z f$ by
\begin{align}
	\cM^z(f)(x) := \sup_{ Q \ni x} \aint_{Q_z} |f(y)|~\mathrm{d}y,
\end{align}
where $Q_z = Q + z \ell(Q)$ and $\ell(Q)$ denotes a side length of $Q$.
The (unweighted) $L_p$-boundedness of $\cM^z$ is well-known (\textit{e.g.} \cite{Heo_Hong_Yang2020,Mus_Sch2013});
$$
\|\cM^z\|_{L_p \to L_p} \lesssim \log(10+|z|).
$$
Here we give a weighted $L_p$-boundedness of $\cM^z$.
\begin{prop}\label{230610919}
    If $w\in A_p(\bR^d)$, then for any $r\in[1,R_{w,p})$ and $s\in[1,\Gamma_w)$, we have 
    \begin{align}\label{smf_wlog_ubound}
        \| \cM^z \|_{L_p(w) \to L_p(w)} 
        \lesssim N(w)
        (1+|z|)^{\frac{d}{p}\left(\frac{p}{r} + \frac{1}{s}-1\right)},
    \end{align}
where $N(w)=[w]_{A_p}^{\frac{1}{p-1}}[w]_{A_{p/r}}^{\frac{1}{p}} [w]_{RH_s}^{\frac{1}{p}}\geq 1$.
The implicit constant depends only on $d,\,p,\,s$, and $r$.
\end{prop}

\begin{proof}
Let $f\in L_p(w)$.
We only prove for the case $z=\vec{n}\in\bZ^d$ and $\vec{n}\not=0$, since generalization for $z\in\bR^d$ follows from the argument in \cite[Lemma 2.9]{Heo_Hong_Yang2020}.
Let $\bQ^\lambda$ be a collection of all dyadic cubes $Q$ such that
\begin{itemize}
	\item $Q$ is maximal with respect to inclusion;
	\item $\frac{1}{|Q|}\int_{Q}|f(y)|~\mathrm{d}y > \lambda$.
\end{itemize}
Then it follows that
\begin{align}
	\bigcup_{Q \in \bQ^\lambda} Q = \{ x\in\bR^d : \cM(f)(x) > \lambda\}.
\end{align}
Moreover, for $m=[\log_2|\vec{n}|]+1$ and each dyadic cube $Q\in \bQ^\lambda$, the cubes 
$$
Q^i=2^{-i}\ell(Q)\vec{n}+Q\,\,\,\,(i=0,\,\dots,\,m-1)\quad\text{and}\quad Q^m=3Q
$$ satisfy
\begin{align}\label{condi_Q}
    \begin{split}
        (a)\,\, &\text{$\ell(Q^i)\leq 3\ell(Q)$};\\
        (b)\,\, &\text{$Q^i \subset 2^{-i+3}|\vec{n}|Q$};\\
        (c)\,\, &\{x\in\bR^d : \cM^{\vec{n}}(f)(x)>\lambda\} \subset \bigcup\limits_{Q \in \bQ^\lambda} \left(Q^0 \cup\cdots\cup Q^{m}\right)
    \end{split}
\end{align}
(see the proofs of Lemmas 2.6 and 2.9 in \cite{Heo_Hong_Yang2020}).
Here, $aQ$ ($a\in\bR_+$) denotes a cube whose side length is $a$ times of $\ell(Q)$ with the same center as $Q$.
From \eqref{condi_Q}, we have
\begin{align*}
w(Q^i)&=\frac{w(Q^i)}{w(2^{-i+3}|\vec{n}|Q)}\frac{w(2^{-i+3}|\vec{n}|Q)}{w(Q)}w(Q)\\
&\leq [w]_{RH_s}\left(\frac{|Q^i|}{(2^{-i+3}|\vec{n}|)^d|Q|}\right)^{1-1/s}[w]_{A_{p/r}}\left(\frac{(2^{-i+3}|\vec{n}|)^d|Q|}{|Q|}\right)^{p/r}w(Q)\\
&\lesssim_d [w]_{A_{p/r}}[w]_{RH_s}(2^{-i+3}|\vec{n}|)^{\left(\frac{p}{r} + \frac{1}{s}-1\right)d}w(Q)
\end{align*}
(see Propositions~\ref{23.04.21.16.12}.$(ii)$ and $(v)$) and
\begin{equation}\label{ineq_230321_2038}
	\begin{alignedat}{2}
	&&&w\Big(\{x\in\bR^d : \cM^{\vec{n}}(f)(x)>\lambda\}\Big)\\
	&\leq &&\sum_{Q \in \bQ^\lambda}\sum_{i=0}^m w(Q^i)\\
	&\lesssim_d&&[w]_{A_{p/r}}[w]_{RH_s}\left(\sum_{i=0}^m2^{(-l+3)(p+1/s-1)d}\right)|\vec{n}|^{\left(\frac{p}{r} + \frac{1}{s}-1\right)d}\sum_{Q \in \bQ_n^\lambda}^m w(Q)\\
	&\lesssim_{d,p,r,s}&&[w]_{A_{p/r}}[w]_{RH_s}\cdot |\vec{n}|^{\left(\frac{p}{r} + \frac{1}{s}-1\right)d}\cdot w\left( \{ x\in\bR^d : \cM(f)(x) > \lambda\} \right)\,.
	\end{alignedat}
\end{equation}
Thus we have
\begin{align*}
    \|\cM^{\vec{n}}f\|_{L_p(w)}^p
    &= p\int_0^\infty \lambda^{p-1} w\left(\{\cM^{\vec{n}}f>\lambda\}\right)~\mathrm{d}\lambda\\
    &\lesssim p [w]_{A_{p/r}}[w]_{RH_s}|{\vec{n}}|^{\left(\frac{p}{r} + \frac{1}{s}-1\right)d}\int_0^\infty \lambda^{p-1} w\left(\{\cM f>\lambda\}\right)~\mathrm{d}\lambda\\
    &= [w]_{A_{p/r}} [w]_{RH_s}|\vec{n}|^{\left(\frac{p}{r} + \frac{1}{s}-1\right)d} \|\cM f\|_{L_p(w)}^p
\end{align*}
and this means that
\begin{alignat*}{2}
	\| \cM^{\vec{n}} \|_{L_p(w) \to L_p(w)} &\lesssim_{d,p,r,s} &&[w]_{A_{p/r}}^{\frac{1}{p}} [w]_{RH_s}^{\frac{1}{p}}|\vec{n}|^{\frac{d}{p}\left(\frac{p}{r} + \frac{1}{s}-1\right)}\| \cM \|_{L_p(w) \to L_p(w)}\\
	&\lesssim_{d,p}&&[w]_{A_p}^{\frac{1}{p-1}}[w]_{A_{p/r}}^{\frac{1}{p}} [w]_{RH_s}^{\frac{1}{p}}|\vec{n}|^{\frac{d}{p}\left(\frac{p}{r} + \frac{1}{s}-1\right)}
\end{alignat*}
(for the last inequality, see \cite[Theorem 9.1.9]{grafakos2014modern}).
For $z\in\bR^d$, it is known in the proof of \cite[Lemma 2.9]{Heo_Hong_Yang2020} that for $z=\vec{n}+z_0$, $z_0\in[0,1]^d$
$$
    \cM^z(f)(x) \lesssim \sum_{i=1}^{2^d} \cM^{\vec{n}+\vec{n}_i}(f)(x),
$$
where $\vec{n}_i \in \bZ^d$ and each component of $\vec{n}_i$ takes on values of either $0$ or $1$.
Therefore we have
\begin{align*}
    \| \cM^z \|_{L_p(w) \to L_p(w)} \lesssim_{d,p,r,s} [w]_{A_p}^{\frac{1}{p-1}}[w]_{A_{p/r}}^{\frac{1}{p}} [w]_{RH_s}^{\frac{1}{p}} (1+|z|)^{\frac{d}{p}\left(\frac{p}{r} + \frac{1}{s}-1\right)}.
\end{align*}
The proposition is proved.
\end{proof}

We make use of \eqref{smf_wlog_ubound} to obtain the following lemma:
\begin{lem}\label{lem_wlog_besov}
Let $L\in\mathbb{N}$, $w\in A_{p}(\bR^d)$ and $f\in L_p(\mathbb{R}^d, w\,\mathrm{d}x)$.
Then for any $\rho\in\bR^d$, $r\in[1,R_{w,p})$, and $s\in[1,\Gamma_w)$,
         \begin{align*}
            \| \cD_{\rho}^L S_0f\|_{L_p(w)} 
            &\lesssim N(w) \min(|\rho|,1)^L\big(1+|\rho|\big)^{\frac{d}{p}\left(\frac{p}{r}+\frac{1}{s}-1\right)}\|S_0 f\|_{L_p(w)},\\
         	\| \cD_{\rho}^L\Delta_jf\|_{L_p(w)} 
            &\lesssim N(w)\min(2^j|\rho|,1)^L\big(1+2^j|\rho|\big)^{\frac{d}{p}\left(\frac{p}{r}+\frac{1}{s}-1\right)} \|\Delta_j f\|_{L_p(w)},
         \end{align*}
         where $N(w)$ is the constant in Proposition~\ref{230610919}.
         The implicit constants depend only on $d,p$ and $L$.
\end{lem}

\begin{proof}
	The proof of the first inequality is nothing but taking $j=0$ in the
	proof of the second inequality.
	Hence we only consider $\cD_\rho^L \Delta_j f$.
	For a smooth function $F$ on $\bR^d$, the fundamental theorem of calculus implies that
	\begin{equation}\label{ineq_230322_1843}
		\begin{aligned}
			\cD_\rho F (x) &=F(x+\rho)-F(x)\\
			&= \int_0^1\left( \rho_1D_{x_1}F+\cdots \rho_dD_{x_d}F\right) \left( x+t \rho \right)\,\mathrm{d} t\\
			&\leq |\rho|\sum_{i=1}^d\int_0^1\big|\big(D_{x_i}F\big)(x+t\rho)\big|\,\mathrm{d} t\,.
		\end{aligned}
	\end{equation}
By using \eqref{ineq_230322_1843} inductively, we obtain
\begin{equation}
\label{230610943}
\begin{aligned}
&\cD_\rho^L \Delta_jf (x)\\
&\lesssim_{d,L} |\rho|^L\sum_{\substack{\alpha=(\alpha_1,\ldots,\alpha_d)\in \bN_0^d\\\alpha_1+\cdots+\alpha_d=L}}\int_{[0,1]^L}\big|\big(D_x^\alpha \Delta_jf\big)\big(x+(t_1+\cdots+t_L)\rho\big)\big|\,\mathrm{d}t_1\cdots\mathrm{d}t_L\,,
\end{aligned}
\end{equation}
where $D_x^\alpha=\big(D_{x_1}\big)^{\alpha_1}\cdots \big(D_{x_d}\big)^{\alpha_d}$.

Take a Schwartz function $\phi\in \cS(\bR^d)$ such that $\widehat{\phi}\equiv 1$ on $\supp (\widehat{\psi})$, and denote $\phi_{2^j}(x):=2^{jd}\phi(2^jx)$.
Then, we have $\widehat{\psi_j} = \widehat{\phi_{2^j}}\,\widehat{\psi_j}$, 
which implies that $\Delta_j f = \phi_{2^j} * (\Delta_j f)$.
Note that for $\alpha$ in \eqref{230610943},
$$
D_{x}^{\alpha}\Delta_jf = \big(D_{x}^{\alpha}\phi_{2^j}\big)\ast \Delta_jf\quad\text{and}\quad D_{x}^{\alpha}\phi_{2^j}=2^{jL} 2^{jd} \big(D_{x}^\alpha \phi\big)(2^j \,\cdot\,).
$$
Since $D_x^\alpha\phi\in\cS(\bR^d)$, we obtain that 
$$
\big|D_x^\alpha\phi\big(x)|\lesssim_{\alpha,d}(1+|x|)^{-100 d}
$$
and that for any $\rho\in\bR^d$,
	\begin{equation*}
	\begin{alignedat}{2}
		|D_{x}^{\alpha}\Delta_jf(x + \rho)| 
		&\lesssim_{\alpha,d}\,&& 2^{jL}\int_{\bR^d} |\Delta_jf(y)| \frac{2^{jd}}{(1+2^j|x-y+\rho|)^{100d}} ~\mathrm{d}y\\
		&\lesssim_d &&2^{jL}\sum_{\vec{n}\in\bZ^d} (1+|\vec{n}|)^{-100d} 2^{jd} \int_{(x+\rho + 2^{-j} \vec{n})+I_j} |\Delta_jf(y)|~\mathrm{d}y\\
		&\leq &&2^{jL}\sum_{\vec{n}\in\bZ^d} (1+|\vec{n}|)^{-100d} \sum_{k=1}^{N_d}2^{jd} \int_{(\rho + 2^{-j} \vec{n})+I_{j,k}} |\Delta_jf(y)|~\mathrm{d}y,
	\end{alignedat}
\end{equation*}
	where $I_j = [0, 2^{-j}]^d$ and $\{I_{j,k}\}_{k=1,\dots,N_d}$ is a collection of all dyadic cubes such that $l(I_{j,k})=l(I_j)$ and $(x+I_j)\cap I_{j,k}\neq \emptyset$.
	Thus we have
$$
		|D_{x}^{\alpha}\Delta_jf(x + \rho)| 
		\lesssim_{\alpha,d}
		2^{jL}\sum_{\vec{n}\in\bZ^d} (1+|\vec{n}|)^{-100d} \cM^{2^j \rho+\vec{n}}\left(\Delta_jf\right)(x)
$$
and by Proposition~\ref{230610919}, we obtain
\begin{align}\label{ineq_230322_1845}
	\|D_{x}^{\alpha}\Delta_jf(\,\cdot + \rho)\|_{L_p(w)} \lesssim N(w) 
	2^{jL}\sum_{\vec{n}\in\bZ^d} (1+|\vec{n}|)^{-100d} (1+2^j|\rho|+|\vec{n}|)^{\frac{d}{r}+\frac{d}{ps}-\frac{d}{p}}\|\Delta_jf\|_{L_p(w)}\,.
\end{align}
	With help of \eqref{230610943} and \eqref{ineq_230322_1845}, it follows that
	\begin{equation}\label{ineq_230321_2126}
		\begin{aligned}
			\|\cD_{\rho}^L \Delta_jf\|_{L_p(w)}&\lesssim  
			N(w)\big(2^{j}|\rho|\big)^L\sum_{\vec{n}\in\bZ^d}
			\frac{(1+ |\vec{n}|+2^j|L\rho|)^{\frac{d}{r} + \frac{d}{ps}-\frac{d}{p}}}{(1+|\vec{n}|)^{100d}}
			\|\Delta_j f\|_{L_p(w)}\\
			& \lesssim N(w)\big(2^{j}|\rho|\big)^L\big(1+2^j|\rho|\big)^{\frac{d}{r} + \frac{d}{ps}-\frac{d}{p}}\|\Delta_j f\|_{L_p(w)},
		\end{aligned}
	\end{equation}
and note that $\frac{d}{r} + \frac{d}{ps}-\frac{d}{p}\leq d$.
	
On the other hand, we have that
	\begin{align*}
		\cD_\rho^L\Delta_j f (x)=\sum_{k=0}^{L}\binom{L}{k}(-1)^{L-k}\Delta_jf(x+k\rho).
	\end{align*}
	For each term in the above, one can check from \eqref{ineq_230322_1845} that
	\begin{align*}
		\|\Delta_jf(x+k\rho)\|_{L_p(w)} &\lesssim N(w) (1+ 2^{j}k|\rho|)^{\frac{d}{r} + \frac{d}{ps}-\frac{d}{p}} \|\Delta_j f\|_{L_p(w)}\,,
	\end{align*}
	which implies 
	\begin{align}\label{ineq_wlog2}
		\|\cD_{\rho}^L\Delta_jf\|_{L_p(w)} \lesssim N(w) (1+ 2^{j}|\rho|)^{\frac{d}{r} + \frac{d}{ps}-\frac{d}{p}} \|\Delta_j f\|_{L_p(w)}.
	\end{align}
	The lemma is proved by a combination of \eqref{ineq_230321_2126} and \eqref{ineq_wlog2}.
\end{proof}

\subsection{Peetre's maximal functions}
Peetre's maximal functions are defined by
$$
    M_{j,\lambda}^{*}f(x):=\sup_{y\in\bR^d}\frac{|\Delta_jf(x-y)|}{(1+2^j|y|)^{\lambda}},\quad 
    M_{\lambda}^{0,*}f(x):=\sup_{y\in\bR^d}\frac{|S_0f(x-y)|}{(1+|y|)^{\lambda}},
$$
where $j\in\mathbb{Z}$ and $\lambda>0$.
\begin{lem}
\label{22.12.28.17.28}
    Let $f\in\cS'(\bR^d)$.
    \begin{enumerate}[(i)]
        \item 
 			For any $r\in(0,\infty)$ and $\varphi\in \cS(\bR^d)$ with $\mathrm{supp}(\widehat{\varphi})\subseteq B_2(0)$, there exists $N=N(d,r)>0$ such that
            \begin{align*}
            	\sup_{z\in\bR^d}\frac{|\varphi(x-z)|}{(1+|z|)^{d/r}}\lesssim_{d,r}(\cM(|\varphi|^r)(x))^{1/r}.
            \end{align*}
        
        \item
            For any $j\in\mathbb{Z}$, $\lambda>0$ and multi-index $\alpha$,
            \begin{align*}
            	|D^{\alpha}\Delta_jf(x)| &\lesssim_{d,\lambda,\alpha}2^{j|\alpha|}M_{j,\lambda}^{*}f(x),\\
            	|D^{\alpha}S_0f(x)| &\lesssim_{d,\lambda,\alpha}M_{\lambda}^{0,*}f(x).
            \end{align*}
        
        \item 
            For any $r\in(0,\infty)$, 
            \begin{align*}
            	M_{j,d/r}^{*}f(x) &\lesssim_{d,r} (\cM(|\Delta_jf|^r)(x))^{1/r},\\
            	M_{d/r}^{0,*}f(x) &\lesssim_{d,r} (\cM(|S_0f|^r)(x))^{1/r}
            \end{align*}
        
        \item 
            For any $\rho\in\bR^d$ and $L\in\bN$,
            \begin{align*}
            	|\cD_{\rho}^L\Delta_jf(x)| &\lesssim_{d,L} \min(1,|\rho|2^{j})^L(1+|\rho|2^j)^{\lambda}M_{j,\lambda}^{*}f(x),\\
            	|\cD_{\rho}^LS_0f(x)| &\lesssim_{d,L}\min(1,|\rho|)^L(1+|\rho|)^{\lambda}M_{\lambda}^{0,*}f(x)
            \end{align*}
        
        \item 
            Suppose that $f$ is also a locally integrable function.
            For any $j\in\bZ$,
            \begin{align*}
            	|S_0f(x)|+|\Delta_jf(x)|\lesssim_d \cM f(x).
            \end{align*}
    \end{enumerate}
\end{lem}

\begin{proof}
    \begin{enumerate}[(i)]
        \item 
            This is an easy consequence of \cite[Theorem 1.3.1-1]{triebel1983theory} with the fact that $(1+|z|)^{d/r}\simeq 1+|z|^{d/r}$.
        
        \item
            By an almost orthogonality of $\Delta_j$,
            \begin{align*}
                |D^{\alpha}_{x}\Delta_jf(x)|&\leq \int_{\bR^d}|\Delta_jf(x-y)|\sum_{k=-1}^1|D^{\alpha}_{x}\psi_{2^{j+k}}(y)|~\mathrm{d}y \\
                &\leq M_{j,\lambda}^{*}f(x)\int_{\bR^d}(1+2^j|y|)^{\lambda}\sum_{k=-1}^1|D^{\alpha}_{x}\psi_{2^{j+k}}(y)|~\mathrm{d}y.
            \end{align*}
            Clearly,
            $$
            	\int_{\bR^d}(1+2^j|y|)^{\lambda}\sum_{k=-1}^1|D^{\alpha}_{x}\psi_{2^{j+k}}(y)|~\mathrm{d}y=:c2^{j|\alpha|},
            $$
            and we remark that $c$ is independent of $j$. 
            Similarly,
            $$
            	|D^{\alpha}_{x}S_0f(x)|\leq M_{\lambda}^{0,*}f(x).
            $$
        
        \item 
            Let $g_j(x):=\Delta_jf(2^{-j}x)$. Then
            $$
            	\widehat{g_j}(\xi):=2^{jd}\widehat{\psi}(\xi)\widehat{f}(2^j\xi),\quad \widehat{S_0f}(\xi):=\widehat{\Phi}(\xi)\widehat{f}(\xi).
            $$
            Thus $\supp(\widehat{g_j}),\supp(\widehat{S_0f})\subseteq B_2(0)$. 
            Then by the first assertion,
            \begin{align*}
                \sup_{z\in\bR^d}\frac{|g_j(x-z)|}{(1+|z|)^{d/r}} &\leq N(\cM(|g_j|^r)(x))^{1/r},\\
                \sup_{z\in\bR^d}\frac{|S_0f(x-z)|}{(1+|z|)^{d/r}} &=M_{\lambda}^{0,*}f(x) \leq N(\cM(|S_0f|^r)(x))^{1/r}.
            \end{align*}
            This implies that
            \begin{align*}
                M_{j,d/r}^{*}f(2^{-j}x)&=\sup_{z\in\bR^d}\frac{|\Delta_jf(2^{-j}x-z)|}{(1+2^{j}|z|)^{d/r}}=\sup_{z\in\bR^d}\frac{|\Delta_jf(2^{-j}x-2^{-j}z)|}{(1+|z|)^{d/r}}\\
                &=\sup_{z\in\bR^d}\frac{|g_j(x-z)|}{(1+|z|)^{d/r}}\\
                &\leq N(\cM(|g_j|^r)(x))^{1/r}=N(\cM(|\Delta_jf|^r)(2^{-j}x))^{1/r}.
            \end{align*}
        
        \item 
          	Due to \eqref{230610943} and the definition of $\cD_\rho^L$, we have
            $$
            	|\cD_{\rho}^L\Delta_jf(x)|\leq \min\Big(|\rho|^{L}\sup_{|z|\leq L|\rho|}\sum_{|\alpha|=L}|D^{\alpha}_{x}\Delta_jf(x-z)|,\sup_{|z|\leq L|\rho|}|\Delta_jf(x-z)|\Big).
            $$
            By the second assertion of the lemma, 
            $$
            	|\cD_{\rho}^L\Delta_jf(x)|\leq N\min(1,|\rho|2^{j})^L\sup_{|z|\leq L|\rho|}M_{j,\lambda}^{*}f(x-z).
            $$
            Since
            $$
            	\sup_{|z|\leq L|\rho|}M_{j,\lambda}^{*}f(x-z)\leq N(1+2^{j}|\rho|)^{\lambda}M_{j,\lambda}^{*}f(x),
            $$
            we obtain the result. 
            Similarly,
            $$
            	|\cD_{\rho}^{L}S_0f(x)|\leq N\min(1,|\rho|)^L\sup_{|z|\leq L|\rho|}M_{\lambda}^{0,*}f(x-z)
            $$
            and
            $$
            	\sup_{|z|\leq L|\rho|}M_{\lambda}^{0,*}f(x-z)\leq N(1+|\rho|)^{\lambda}M_{\lambda}^{0,*}f(x).
            $$
        
        \item 
            Note that $\Phi$ from \eqref{ineq_230315_1246} is in $\cS(\bR^d)$, hence it follows that
            \begin{align*}
            	|S_0f(x)|
		\leq &\int_{B_{1}(0)}|f(x-y)\Phi(y)|~\mathrm{d}y\\
		 &+\sum_{k=1}^{\infty}\int_{B_{2^k}(0)\setminus\overline{B_{2^{k-1}}(0)}}|f(x-y)\Phi(y)|~\mathrm{d}y
            \end{align*}
            and
            $$
            	|\Phi(y)|+|y|^{d+1}|\Phi(y)|\leq N.
            $$
            This certainly implies that $|S_0f|\leq N\cM f$. For $\Delta_j$, the change of variable formula yields
            $$
            	\Delta_jf(x)
		=\int_{\bR^d}f(x-y)\psi_{2^j}(y)~\mathrm{d}y
		=\int_{\bR^d}f(x-2^{-j}y)\psi(y)~\mathrm{d}y.
            $$
            Using the similar arguments, we have $|\Delta_jf(x)|\leq N\cM f(x)$. The lemma is proved.
    \end{enumerate}
\end{proof}

\subsection{Vector-valued inequalities}
\begin{defn}
    An operator $T:L_p(\mathbb{R}^d,w\,\mathrm{d}x)\to L_p(\mathbb{R}^d,w\,\mathrm{d}x)$ is linearizable if there exist a Banach space $B$ and a $B$-valued linear operator $U:L_p(\mathbb{R}^d,w\,\mathrm{d}x)\to L_p(\mathbb{R}^d,w\,\mathrm{d}x;B)$ such that
    $$
    	|Tf(x)| = \| Uf(x)\|_B,\quad \forall f\in L_p(\mathbb{R}^d,w\,\mathrm{d}x).
    $$
\end{defn}

\begin{lem}(\cite[p. 521, Remark 6.5]{RdF1985weighted}) 
\label{23.01.31.17.22}
    Let $\{T_j\}_{j\in\bZ}$ be a sequence of linearizable operators satisfy, for some fixed $r>1$,
    	\begin{equation}\label{22.09.13.16.42}
            \sup_{j\in\bZ}\int_{\bR^d} | T_j f(x)|^r w(x) ~\mathrm{d}x \leq C(K) \int_{\bR^d} |f(x)|^r w(x) ~\mathrm{d}x,
        \end{equation}
        for all $w\in A_r(\bR^d)$, $f\in L_r(\bR^d,w\,\mathrm{d}x)$,
        where $[w]_{A_r(\bR^d)}\leq K$. Then for $1<p,q<\infty$ we have
        $$
            \left\| \left( \sum_{j\in\bZ} |T_j G_j|^q \right)^{1/q} \right\|_{L_p(\bR^d,w'\,dx)} 
            \leq C \left\| \left( \sum_{j\in\bZ} |G_j|^q\right)^{1/q} \right\|_{L_p(\bR^d,w'\,dx)},
        $$
        for all $w'\in A_p(\bR^d)$, $\{G_j\}_{j\in\mathbb{Z}}\in L_p(\mathbb{R}^d,w'\,\mathrm{d}x;\ell_q)$,
        where $[w']_{A_p(\bR^d)}\leq K'$ and $C=C(p,q,K')$.
\end{lem}

\begin{lem}\label{Lp_bound}
    The spaces $B_{p,q}^{\phi}(\bR^d,w\,\mathrm{d}x)$ and $F_{p,q}^{\phi}(\bR^d,w\,\mathrm{d}x)$ are continuously embedded into $L_p(\bR^d,w\,\mathrm{d}x)$;
    $$
    	\|f\|_{L_p(\bR^d,w\,\mathrm{d}x)}\leq N\min(\|f\|_{B_{p,q}^{\phi}(\bR^d,w\,\mathrm{d}x)},\|f\|_{F_{p,q}^{\phi}(\bR^d,w\,\mathrm{d}x)}).
    $$
\end{lem}

\begin{proof}
    One can observe that $f=S_0f+\sum_{j=1}^{\infty}\Delta_jf$.
    By Minkowski's inequality,
    $$
    	\|f\|_{L_p(\bR^d,w\,\mathrm{d}x)}\leq \|S_0f\|_{L_p(\bR^d,w\,\mathrm{d}x)}+\sum_{j=1}^{\infty}\|\Delta_jf\|_{L_p(\bR^d,w\,\mathrm{d}x)}.
    $$
    By \eqref{w scaling}, we have
    \begin{equation}
    \label{23.01.29.17.13}
        2^{\varepsilon j}\lesssim\frac{\phi(2^j)}{\phi(1)},
    \end{equation}
    which clearly implies that  $\sum_{j=1}^{\infty}\phi(2^j)^{-1}\lesssim \sum_{j=1}^{\infty}2^{-\varepsilon j}<\infty$.
    Hence
    $$
    \sum_{j=1}^{\infty}\|\Delta_jf\|_{L_p(\mathbb{R}^d,w\,\mathrm{d}x)}=\sum_{j=1}^{\infty}\phi(2^j)^{-1}\phi(2^j)\|\Delta_jf\|_{L_p(\mathbb{R}^d,w\,\mathrm{d}x)}\lesssim \sup_{j\in\mathbb{N}}\phi(2^j)\|\Delta_jf\|_{L_p(\mathbb{R}^d,w\,\mathrm{d}x)}.
    $$
    Therefore, $\|f\|_{L_p(\bR^d,w\,\mathrm{d}x)}\leq N\|f\|_{B_{p,\infty}^{\phi}(\bR^d,w\,\mathrm{d}x)}$.
    For the case $q \in [1, \infty)$, it directly follows that
    \begin{align*}
    &\sup_{j\in\mathbb{N}}\phi(2^j)\|\Delta_jf\|_{L_p(\mathbb{R}^d,w\,\mathrm{d}x)}\\
    \leq&\max\left(\left(\sum_{j=1}^{\infty}\phi(2^j)^q\|\Delta_jf\|_{L_p(\mathbb{R}^d,w\,\mathrm{d}x)}^q\right)^{1/q},\left\|\left(\sum_{j=1}^{\infty}\phi(2^j)^q|\Delta_jf|^q\right)^{1/q}\right\|_{L_p(\mathbb{R}^d,w\,\mathrm{d}x)}\right).
    \end{align*}
    and therefore,
    $$
    	\|f\|_{L_p(\bR^d,w\,\mathrm{d}x)}\leq N\min(\|f\|_{B_{p,q}^{\phi}(\bR^d,w\,\mathrm{d}x)},\|f\|_{F_{p,q}^{\phi}(\bR^d,w\,\mathrm{d}x)}).
    $$
The lemma is proved.
\end{proof}

\subsection{Equivalent norms for $B_{p,q}^{\phi}(w)$ and $F_{p,q}^{\phi}(w)$}\label{ssec_equiv}

We introduce various norms equivalent to those of $B_{p,q}^{\phi}(\bR^d, w~\mathrm{d}x)$ and $F_{p,q}^{\phi}(\bR^d, w~\mathrm{d}x)$.
Let $\phi\in\cI_o(0,M)$ and $L$ be the smallest integer greater than $M$. 
Then for $p\in(1,\infty]$, $q\in[1,\infty]$, we define the following quantities: For $q\in[1,\infty)$, we define
\begin{align*}
	B_0 &:= \|f\|_{L_p(\bR^d,w\,\mathrm{d}x)}+ \left( \sum_{j\in\bZ} \phi(2^j)^q \|\Delta_j f \|_{L_p(\bR^d,w\,\mathrm{d}x)}^q\right)^{1/q},\\
	B_1 &:= \inf\left(  \|f_0\|_{L_p(\bR^d,w\,\mathrm{d}x)}+ \left( \sum_{j\geq 0} \phi(2^j)^q \| f - f_j \|_{L_p(\bR^d,w\,\mathrm{d}x)}^q\right)^{1/q} \right),\\
	B_2 &:= \|f\|_{L_p(\bR^d,w\,\mathrm{d}x)}+\left(\int_0^\infty \phi(t^{-1})^q \|\psi_{\frac{1}{t}}\ast f \|_{L_p(\bR^d,w\,\mathrm{d}x)}^q ~\frac{\mathrm{d}t}{t}\right)^{1/q},\\
	B_3 &:= \|f\|_{L_p(\bR^d,w\,\mathrm{d}x)}+\left(\int_0^\infty \phi(t^{-1})^q \|\psi_{\frac{1}{t}}^L \ast f \|_{L_p(\bR^d,w\,\mathrm{d}x)}^q ~\frac{\mathrm{d}t}{t}\right)^{1/q},\\
 B_4 &:= \|M_{d/r}^{0,*}f\|_{L_p(\bR^d,w\,\mathrm{d}x)}+\left( \sum_{j\geq 1} \phi(2^j)^q \|M_{j,d/r}^{*}f\|_{L_p(\bR^d,w\,\mathrm{d}x)}^q \right)^{1/q}\quad \forall r\in(0,R_{w,p}),\\
 B_5 &:= \|f\|_{L_p(\bR^d,w\,\mathrm{d}x)}+\left( \sum_{j\in\bZ} \phi(2^j)^q \|M_{j,d/r}^{*}f\|_{L_p(\bR^d,w\,\mathrm{d}x)}^q \right)^{1/q}\quad \forall r\in(0,R_{w,p}).
\end{align*}
For $q = \infty$, we make the natural modifications by replacing 
$\ell_q$ with $\ell_\infty$ and 
$L_q((0,\infty), \frac{\mathrm{d}t}{t})$ with $L_\infty((0,\infty))$.
Similarly, for $p,q \in (1,\infty)$, we define
\begin{align*}
	F_0 &:= \|f\|_{L_p(\bR^d,w\,\mathrm{d}x)}+ \left\| \left(\sum_{j\in\bZ} \phi(2^j)^q |\Delta_j f |^q\right)^{1/q} \right\|_{L_p(\bR^d,w\,\mathrm{d}x)},\\
	F_1 &:= \inf \left( \|f_0\|_{L_p(\bR^d,w\,\mathrm{d}x)}+ \left\| \left( \sum_{j\geq 0} \phi(2^j)^q | f-f_j|^q \right)^{1/q}\right\|_{L_p(\bR^d,w\,\mathrm{d}x)} \right),\\
	F_2 &:= \|f\|_{L_p(\bR^d,w\,\mathrm{d}x)}+\left\| \left(\int_0^\infty \phi(t^{-1})^q |\psi_{\frac{1}{t}}\ast f |^q ~\frac{\mathrm{d}t}{t}\right)^{1/q}\right\|_{L_p(\bR^d,w\,\mathrm{d}x)},\\
	F_3 &:= \|f\|_{L_p(\bR^d,w\,\mathrm{d}x)}+\left\| \left(\int_0^\infty \phi(t^{-1})^q |\psi_{\frac{1}{t}}^L\ast f |^q ~\frac{\mathrm{d}t}{t}\right)^{1/q}\right\|_{L_p(\bR^d,w\,\mathrm{d}x)},\\
	F_4 &:= \|M_{d/r}^{0,*}f\|_{L_p(\bR^d,w\,\mathrm{d}x)}+\left\| \left( \sum_{j\geq 1} \phi(2^j)^q |M_{j,d/r}^{*}f|^q \right)^{1/q} \right\|_{L_p(\bR^d,w\,\mathrm{d}x)}\quad \forall r\in(0,R_{w,p})\\
 F_5 &:= \|f\|_{L_p(\bR^d,w\,\mathrm{d}x)}+ \left\| \left( \sum_{j\in\bZ} \phi(2^j)^q |M_{j,d/r}^{*}f|^q \right)^{1/q} \right\|_{L_p(\bR^d,w\,\mathrm{d}x)}\quad \forall r\in(0,R_{w,p}).
\end{align*}
Note that the infimums in $B_1$ and $F_1$ are taken over all sequences of functions $\{f_j\}_{j=0}^{\infty}\subseteq \cS'(\bR^d)\cap L_p(\bR^d, w\,\mathrm{d}x)$ which satisfy
$$
f=\lim_{j\to\infty}f_j\quad \textrm{in}\quad \cS'(\bR^d),\quad \textrm{and}\quad \supp(\widehat{f_j})\subseteq B_{2^{j+1}}(0).
$$
The function $\psi_t$ denotes $t^d \psi(t x)$ for $\psi$ chosen for the Littlewood-Paley projection operators.
The function $\psi_{1/t}^L$ appeared in $B_3$ and $F_3$ is defined by 
$$
\psi^L_{1/t}(y):=\sum_{j=1}^{L}\binom{L}{j}(-1)^{L-j}\psi_{1/(tj)}(y).
$$
Recall that the operators $M_{d/r}^{0,*}$ and $M_{j,d/r}^{*}$ in $B_4$ and $F_4$ are Peetre's maximal operators defined by
$$
M_{j,\lambda}^{*}f(x):=\sup_{y\in\bR^d}\frac{|\Delta_jf(x-y)|}{(1+2^j|y|)^{\lambda}},\quad M_{\lambda}^{0,*}f(x):=\sup_{y\in\bR^d}\frac{|S_0f(x-y)|}{(1+|y|)^{\lambda}}.
$$
We introduce the following proposition, which will be used in the proof of Theorem~\ref{thm_lp_diff_equiv}:
\begin{prop}
\label{prop_wBesov_equiv}
Let $\phi\in\cI_o(0,M)$ and $L\geq M$ be a natural number. 
\begin{enumerate}
\item (Weighted Besov space) For $p\in(1,\infty]$, $q\in[1,\infty]$, we have
\begin{align*}
	\|f\|_{B_{p,q}^{\phi}(\bR^d,w\,\mathrm{d}x)} \simeq B_i\quad \text{for}\quad i=0,1,2,3,4,5.
\end{align*}
\item (Weighted Triebel-Lizorkin space) For $p,q\in(1,\infty)$, we have
\begin{align*}
	\|f\|_{F_{p,q}^{\phi}(\bR^d,w\,\mathrm{d}x)} \simeq F_i\quad \text{for} \quad i=0,1,2,3,4,5.
\end{align*}
\end{enumerate}
\end{prop}

The proof of Proposition~\ref{prop_wBesov_equiv} is given by showing the following relations:
Let $X=B$ or $F$.
Then in Section~\ref{sec_prop_equiv}, we show 
\begin{eqnarray}
	\|f\|_{X_{p,q}^\phi(\bR^d, w~\mathrm{d}x)} 
	\left\{
	\begin{array}{llll}
	\underset{\text{Lemma~\ref{23.02.08.15.47}}}{\simeq} X_0,\\
	\underset{\text{Lemma~\ref{23.02.08.15.52}}}{\simeq} X_1,\\
	\underset{\text{Lemma~\ref{23.02.08.15.54}}}{\simeq} X_2 \underset{\text{Lemma~\ref{23.02.08.15.55}}}{\simeq} X_3,\\
	\underset{\text{Lemma~\ref{23.02.09.14.43}}}{\simeq} X_4, X_5
	\end{array}
	\right.
\end{eqnarray}

In the proof of Theorem \ref{thm_lp_diff_equiv}, we take Proposition \ref{prop_wBesov_equiv} for granted.
In particular, the characterizations of $X$ by $X_3$ and $X_5$ play an essential role in the proof.
For the left of paper, we simply write $X_{p,q}^\phi(w)$ and $L_p(w)$ to denote $X_{p,q}^\phi(\bR^d, w~\mathrm{d}x)$ and $L_p(\bR^d, w~\mathrm{d}x)$, if there is no confusion.
Moreover, we always assume for $X_{p,q}^\phi(w)$ that $p\in(1,\infty],q\in[1,\infty]$ if $X=B$, and $p,q\in(1,\infty)$ if $X=F$.

\mysection{Proof of  Theorem~\ref{thm_lp_diff_equiv}}\label{sec_mainproof}

To begin the proof of Theorem~\ref{thm_lp_diff_equiv},
we introduce quantities $\cB_0$ and $\cF_0$ related to norms of $\cB_{p,q}^\phi(w)$ and $\cF_{p,q}^\phi(w)$, respectively:
\begin{eqnarray*}
    \cB_0:=\|f\|_{L_p(w)}
    +\left\{\begin{array}{ll} \left(\int_{|h|\leq 1} \phi(|h|^{-1})^q  \sup_{|\rho|\leq|h|}\left\| \cD_{\rho}^L f \right\|_{L_p(w)}^q~\frac{\mathrm{d}h}{|h|^d}\right)^{1/q}, &q<\infty,\\
    \sup_{|h|\leq 1}\phi(|h|^{-1}) \sup_{|\rho|\leq|h|}\left\| \cD_{\rho}^L f \right\|_{L_p(w)},&q=\infty.
    \end{array}
    \right.
\end{eqnarray*}
\begin{eqnarray*}
    \cF_0:=\|f\|_{L_p(w)}
    + \left\| \left(\int_{|h|\leq 1} \phi(|h|^{-1})^q\sup_{|\rho|\leq|h|}   |\cD_{\rho}^L (f)(x)|~\frac{\mathrm{d}h}{|h|^d}\right)^{1/q} \right\|_{L_p(w)}^q,
\end{eqnarray*}
where $w\in A_p(\bR^d)$, $\phi\in\cI_o(d/{R_{w,p}}+d/{p\Gamma_w}-d/{p},M)$ for $\cB_0$, $\phi\in\cI_o(d/R_{w,p},M)$ for $\cF_0$ and $L$ is the smallest integer not less than $M$.
We also define quantities $\cB_1$ and $\cF_1$ with $\bR^d$ instead of $|h|\leq 1$, which is related to $\mathscr{B}_{p,q}^\phi(w)$ and $\mathscr{F}_{p,q}^\phi(w)$, respectively.
Then, due to Proposition~\ref{prop_wBesov_equiv}, our strategy is to show that for $\cX = \cB$ or $\cF$
\begin{align}\label{ineq_main}
	X_3 
    \lesssim  \|f\|_{\cX_{p,q}^{\phi}(w)}
    \lesssim\cX_0  
    \leq \cX_1
    \lesssim X_5.
\end{align}
Note that the second and the third inequalities of \eqref{ineq_main} follow from the definitions of $\cB_0$, $\cF_0$, $\cB_1$ and $\cF_1$.
Also note that \eqref{ineq_main} yields the following inequalities:
$$
    \|f\|_{\cX_{p,q}^{\phi}(w)} 
    \leq \|f\|_{\mathscr{X}_{p,q}^{\phi}(w)}
    \leq \mathcal{X}_1
    \lesssim X_5,\quad \text{$\mathscr{X}=\mathscr{B}$ or $\mathscr{F}$}.
$$
Thus it suffices to verify the first and the last inequalities of \eqref{ineq_main}, which are Lemmas~\ref{23.02.08.15.58} and \ref{23.02.09.14.46}.
\begin{lem}
\label{23.02.08.15.58}
$X_3 \lesssim \|f\|_{\cX_{p,q}^{\phi}(w)}$.
That is, $\cX_{p,q}^{\phi}(w)$ is continuously embedded into the space equipped with the norm $X_3$.
\end{lem}

We remark that Lemma \ref{23.02.08.15.58} holds for all $\phi\in\cI_o(0,M)$. Therefore, we prove Lemma \ref{23.02.08.15.58} with the assumption $\phi\in\cI_o(0,M)$.

\begin{lem}
\label{23.02.09.14.46}
$ \cX_1 \lesssim X_5$.
That is, the space equipped with the norm $X_5$ is continuously embedded into the space equipped with the norm $\cX_0$.
\end{lem}
Theorem~\ref{thm_lp_diff_equiv} follows from Lemmas~\ref{23.02.08.15.58}, \ref{23.02.09.14.46} and the equivalences of Proposition~\ref{prop_wBesov_equiv}.

\subsection{Proof of Lemma~\ref{23.02.08.15.58}}

By similarity, we only prove for $X=B$ and $\cX=\cB$.
Recall that 
$$
	B_3 = \|f\|_{L_p(w)} + \left(\int_0^\infty \phi(t^{-1})^q \left\| \psi_{\frac{1}{t}}^{L}\ast f \right\|_{L_p(w)}^q ~\frac{\mathrm{d}t}{t}\right)^{1/q}.
$$
Since $\cX_0$ already contains $\|f\|_{L_p(w)}$, it suffices to show that
\begin{align}
	\int_0^\infty \phi(t^{-1})^q \left\| \psi_{\frac{1}{t}}^{L}\ast f \right\|_{L_p(w)}^q ~\frac{\mathrm{d}t}{t}
	\lesssim \cB_0.
\end{align}
We consider $\int_0^1(\cdots)$, $\int_1^\infty (\cdots)$ and $q\in [1,\infty)$, $q=\infty$ separately.

\vspace{3mm}

\textbf{Case 1.} Estimation of $\int_1^\infty(\cdots)$ when $q\in [1,\infty)$.

For $\int_1^\infty(\cdots)$ and $q\in [1,\infty)$, we want to show 
\begin{align}\label{ineq_230315_1912}
	 \left|\psi_t^L\ast f(x)\right| \lesssim \cM f(x).
\end{align}
Once we assume \eqref{ineq_230315_1912}, then we have
\begin{align}\label{ineq_lower_1}
	\int_{1}^{\infty}\phi(t^{-1})^q\left\|\psi_{1/t}^L\ast f\right\|_{L_p(w)}^q\frac{\mathrm{d}t}{t}
	\lesssim \|f\|_{L_p(\bR^d,w\,\mathrm{d}x)}^q,
\end{align}
Indeed, the inequality \eqref{ineq_lower_1} is valid since $\phi \in \cI_o(0,M)$ and we have
$$
	\int_1^\infty \phi(t^{-1})^q~\frac{\mathrm{d}t}{t} <\infty.
$$
Thus it is left to us obtain \eqref{ineq_230315_1912}.
Observe that $\psi^L_1\in\cS(\bR^d)$ and therefore for any $N\in\bN$, 
$$
\psi^L_{t}(y)=t^d\psi^L_1(ty)\lesssim_{\psi,d,N} \frac{t^d}{(1+|ty|)^{N}}\lesssim_{d,N}\sum_{k=0}^\infty2^{-kN}t^d\bI_{B(0,2^k/t)}(y)\,.
$$
This implies that
\begin{alignat*}{2}
|\psi_t^L\ast f|&\leq \int_{\bR^d}|\psi_t^L (y)||f(x-y)|\dd y\\
 &\lesssim_{d,\psi}\sum_{k=0}^{\infty}2^{-k(d+1)}t^d\int_{\{y:|y|<2^k/t\}}f(x-y)\,\mathrm{d}y\lesssim_d \Big(\sum_{k=0}^\infty2^{-k}\Big)\cM f(x)\,, 
\end{alignat*}
which is \eqref{ineq_230315_1912}.

\vspace{3mm}

\textbf{Case 2.} Estimation of $\int_0^1(\cdots)$ when $q\in[1,\infty)$.

Now we prove
\begin{align}\label{230611958}
    \int_{0}^{1}\phi(t^{-1})^q\left\|\psi_{1/t}^L\ast f\right\|_{L_p(w)}^q\frac{\mathrm{d}t}{t}\lesssim \int_{|y|\leq1}\phi(|y|^{-1})^q\|\cD_y^Lf\|_{L_p(w)}^q~\frac{\mathrm{d}y}{|y|^d}+\|f\|_{L_p(w)}.
\end{align}
Observe that
$$
|\psi_{t}^L\ast f(x)|=\left|\int_{\bR^d}\psi_{t}(y)\cD_y^Lf(x)\dd y\right|\leq \int_{\{y:|y|\leq 1\}}\big|\cdots\big|\dd y+ \int_{\{y:|y|>1\}}\big|\cdots\big|\dd y=:I_{t,1}(x)+I_{t,2}(x)\,.
$$

Due to the Minkowski inequality and H\"older inequality, we have
\begin{align}\label{2306110914}
\begin{split}
\int_{0}^{1}\phi(t^{-1})^q\left\|I_{1/t,1}\right\|_{L_p(w)}^q\frac{\mathrm{d}t}{t}\leq & \int_{0}^{1}\phi(t^{-1})^q\bigg(\int_{\{y:|y|\leq 1\}}|\psi_{1/t}(y)|\big\|\cD_y^Lf\big\|_{L_p(w)}\dd y\bigg)^q\frac{\mathrm{d}t}{t}\\
\leq &\int_{0}^{1}\phi(t^{-1})^q\int_{\{y:|y|\leq 1\}}|\psi_{1/t}(y)|\big\|\cD_y^Lf\big\|_{L_p(w)}^q\dd y\frac{\mathrm{d}t}{t},
\end{split}
\end{align}
where the last inequality is implied by that $\int_{\{y:|y|\leq 1\}}|\psi_{1/t}(y)|\dd y\leq \int_{\bR^d}|\psi(y)|\dd y<\infty$.
Due to Proposition~\ref{23.05.25.13.12}-(v), we have
\begin{align*}
\int_0^1\phi(t^{-1})^q|\psi_{1/t}\left( |y| \right)|\frac{\mathrm{d}t}{t}\,
&=\phi(|y|^{-1})^q|y|^{-d}\int_0^1\left(\frac{\phi(t^{-1})}{\phi(|y|^{-1})}\right)^q \left|\left(\frac{|y|}{t}\right)^{-d}\psi\left(\frac{|y|}{t}\right) \right| \frac{\mathrm{d}t}{t}\\
&\lesssim \phi(|y|^{-1})^q|y|^{-d}\int_0^1 \left\{ \left(\frac{|y|}{t}\right)^{\varepsilon}+\left(\frac{|y|}{t}\right)^L\right\}^q \left(\frac{|y|}{t}\right)^{-d}\left|\psi \left(\frac{|y|}{t}\right)\right|\frac{\mathrm{d}t}{t}\\
&\leq \phi(|y|^{-1})^q|y|^{-d}\int_0^\infty\Big(s^{\varepsilon}+s^L\Big)^qs^d|\psi(s)|\frac{\mathrm{d}s}{s},
\end{align*}
and the integral in the last term is finite.
Consequently, we obtain that
\begin{align}\label{230611957}
\int_{0}^{1}\phi(t^{-1})^q\left\|I_{1/t,1}\right\|_{L_p(w)}^q\frac{\mathrm{d}t}{t}\lesssim \int_{|y|\leq1}\phi(|y|^{-1})^q\|\cD_y^Lf\|_{L_p(w)}^q~\frac{\mathrm{d}y}{|y|^d}.
\end{align}

Since $\psi\in\cS(\bR^d)$, $|\psi(y)|\lesssim |y|^{-(L+d+1)}$, where $L\geq M$ is a natural number and $M$ is the constant in the assumption for $\phi\in \cI_o(d/R_{w,p}+d/(p\Gamma_w)-d/p,M)$.
Therefore for $t\in(0,1)$, we have
\begin{align*}
|I_{1/t,2}(x)|\,&\lesssim t^{-d}\int_{\{y:|y|>1\}}(|y|/t)^{-(L+d+1)}\big|\cD_y^Lf(x)\big|\dd y\\
&\leq \sum_{k:2^kt>1/2} t^{-d}2^{-k(L+d+1)}\int_{\{y:2^kt\leq |y|<2^{k+1}t\}}|\cD_y^Lf(x)|\dd y\\
&\lesssim \sum_{k:2^kt>1/2}2^{-k(L+1)}\cM f(x)\\
&\lesssim t^{L+1}\cM f(x)\,.
\end{align*}
Since $\phi(t^{-1})\lesssim \phi(1) t^{-L}$ for $t<1$ (see Proposition~\ref{23.05.25.13.12}.(v)), 
we have
\begin{align}\label{230611956}
\int_{0}^{1}\phi(t^{-1})^q\left\|I_{1/t,1}\right\|_{L_p(w)}^q\frac{\mathrm{d}t}{t}\lesssim \int_{0}^{1}\phi(t^{-1})^qt^{(L+1)q}\frac{\mathrm{d}t}{t}\|\cM f\|_{L_p(w)}^q\lesssim\|\cM f\|_{L_p(w)}^q\,. 
\end{align}

Consequently \eqref{230611958} follows from \eqref{230611957}, \eqref{230611956}, and \eqref{2306111001}.

\vspace{3mm}

\textbf{Case 3.} Estimations when $q=\infty$.

For $q=\infty$, if we use
$$
\frac{\phi(t^{-1})}{\phi(t^{-1}|y|^{-1})}\lesssim1_{|y|\leq 1}|y|^{\varepsilon}+1_{|y|>1}|y|^{M-\varepsilon},
$$
then by the change of variable formula,
\begin{align*}
    \phi(t^{-1})\left\|\psi_{1/t}^L\ast f\right\|_{L_p(w)}
    &\leq N\phi(t^{-1})\int_{|y|\leq 1/t}|\psi(|y|)|\|\cD_{ty}^Lf\|_{L_p(w)}\mathrm{d}y+N\|f\|_{L_p(w)}\\
    &=N\int_{|y|\leq 1/t}\phi(t^{-1})\psi(|y|)\|\cD_{ty}^Lf\|_{L_p(w)}\mathrm{d}y+N\|f\|_{L_p(w)}\\
    &\leq N\sup_{|y|\leq1}\phi(|y|^{-1})\|\cD_{y}^Lf\|_{L_p(\bR^d,w\,\mathrm{d}x)}+N\|f\|_{L_p(w)}.
\end{align*}
This certainly implies that
$$
	X_3 \lesssim \cX_0, \quad\text{when $q=\infty$}.
$$
The lemma is proved.

\subsection{Proof of Lemma~\ref{23.02.09.14.46}}
In this subsection, we prove Lemma \ref{23.02.09.14.46}.
To prove Lemma \ref{23.02.09.14.46}, it suffices to prove that
\begin{equation}\label{2307101154}
\begin{aligned}
\int_{\bR^d}\phi(|h|^{-1})^q\sup_{|\rho|\leq|h|}\|\cD^L_{\rho}f\|_{L_p(w)}^q\frac{\mathrm{d}h}{|h|^d}
&\lesssim \sum_{j\in\bZ}\phi(2^j)^q\|\Delta_jf\|_{L_p(w)}^q,\\
\left\|\left(\int_{\bR^d}\phi(|h|^{-1})^q\sup_{|\rho|\leq|h|}|\cD^L_{\rho}f|^q\frac{\mathrm{d}h}{|h|^d}\right)^{1/q}\right\|_{L_p(w)}
&\lesssim \left\|\left(\sum_{j\in\bZ}\phi(2^j)^q|M_{j,d/r}^{*}f|^q\right)^{1/q}\right\|_{L_p(w)}.
\end{aligned}
\end{equation}
We divide the proof into two parts.

\subsubsection{Case of $(X,\cX)=(B,\cB)$}

We first prove for $X=B$ and $\cX = \cB$.
One can observe that for $q\in[1,\infty)$
\begin{equation}\label{ineq_230315_2115}
\begin{aligned}
    \int_{\bR^d}\phi(|h|^{-1})^q\sup_{|\rho|\leq|h|}\|\cD_{\rho}^Lf\|_{L_p(w)}^q\frac{\mathrm{d}h}{|h|^d}&\lesssim \int_0^{\infty}\phi(\lambda^{-1})^q\sup_{|\rho|\leq \lambda}\|\cD_{\rho}^Lf\|_{L_p(w)}^q\frac{\mathrm{d}\lambda}{\lambda}\\
    &\lesssim \sum_{k\in\bZ}\phi(2^k)^q\sup_{|\rho|\leq 2^{-k}}\|\cD_{\rho}^Lf\|_{L_p(w)}^q,
\end{aligned}
\end{equation}
and for $q=\infty$
\begin{equation}\label{ineq_230315_2116}
\begin{aligned}
    \sup_{h\in\bR^d}\phi(|h|^{-1})\sup_{|\rho|\leq|h|}\|\cD_{\rho}^Lf\|_{L_p(w)}
    &=\sup_{\lambda\in(0,\infty)}\phi(\lambda^{-1})\sup_{|\rho|\leq\lambda}\|\cD_{\rho}^Lf\|_{L_p(w)}\\
    &\leq N\sup_{k\in\bZ}\phi(2^{k})\sup_{|\rho|\leq 2^{-k}}\|\cD_{\rho}^Lf\|_{L_p(w)}.
\end{aligned}
\end{equation}
Note that we can write $f=\sum_{j\in\bZ}\Delta_jf$, hence by Lemma~\ref{lem_wlog_besov} we have
\begin{equation}\label{ineq_230322_1605}
\begin{aligned}
	\sup_{|\rho|\leq 2^{-k}}\|\cD^L_{\rho}f\|_{L_p(w)}&\leq \sum_{j\in\bZ}\sup_{|\rho|\leq 2^{-k}}\|\cD_{\rho}^L\Delta_jf\|_{L_p(w)}\\
	&\lesssim\sum_{j\in\bZ} \min( 2^{(j-k)L}, 1)(1+ 2^{j-k})^{\frac{d}{r} + \frac{d}{ps}-\frac{d}{p}} \|\Delta_jf\|_{L_p(w)},
\end{aligned}
\end{equation}
where $r\in[1,R_{w,p})$ and $s\in(1,\Gamma_w)$.
We choose $r$ and $s$ later.

Hence by \eqref{ineq_230315_2115}, \eqref{ineq_230315_2116}, and \eqref{ineq_230322_1605}, we have
\begin{align*}
    \int_{\bR^d}\phi(|h|^{-1})^q\sup_{|\rho|\leq|h|}\|\cD_{\rho}^Lf\|_{L_p(w)}^q\frac{\mathrm{d}h}{|h|^d}
    &\lesssim \sum_{k\in\bZ}\phi(2^k)^q\Big(\sum_{j\leq k}^{}  2^{(j-k)L} \|\Delta_jf\|_{L_p(w)}\Big)^q\\
    &\quad+\sum_{k\in\bZ}\phi(2^k)^q\Big(\sum_{j\geq k+1} 2^{(j-k)(\frac{d}{r} + \frac{d}{ps}-\frac{d}{p})} \|\Delta_jf\|_{L_p(w)}\Big)^q\\
    &=:J_{1,q}+J_{2,q},\quad \forall q\in[1,\infty),
\end{align*}
and
\begin{align*}
    \sup_{h\in\bR^d}\phi(|h|^{-1})\sup_{|\rho|\leq|h|}\|\cD_{\rho}^Lf\|_{L_p(w)}
    &\leq \sup_{k\in\bZ}\phi(2^k)\sum_{j=1}^{k} 2^{(j-k)L} \|\Delta_jf\|_{L_p(w)}\\
    &\quad+\sup_{k\in\bZ}\phi(2^k)\sum_{j= k+1}^{\infty}2^{(j-k)(\frac{d}{r} + \frac{d}{ps}-\frac{d}{p})} \|\Delta_jf\|_{L_p(w)}\\
    &=:J_{1,\infty}+J_{2,\infty}.
\end{align*}
Thus the proof for $X=B$ and $\cX = \cB$ is complete if we show the following inequalities:
\begin{align}
J_{1,q}, J_{2,q} &\lesssim \sum_{j\in\bZ}\phi(2^j)^q\|\Delta_jf\|_{L_p(w)}^q, \quad q\in [1,\infty),\\
J_{1,\infty}, J_{2,\infty} &\lesssim \sup_{j\in\bZ}\phi(2^j)\|\Delta_jf\|_{L_p(w)},\quad q=\infty.
\end{align}

\vspace{3mm}

\textbf{Case 1.} Estimation of $J_{1,q}$.

For $q\in[1,\infty)$, by H\"older's inequality it follows that
\begin{align*}
    \Big(\sum_{j\leq k} 2^{(j-k)L} \|\Delta_jf\|_{L_p(w)}\Big)^q
    &\leq N \sum_{j\leq k}2^{q(j-k)(L-\delta)}\|\Delta_jf\|_{L_p(w)}^q.
\end{align*}
%Note that we first multiply and divide $2^{(j-k)\delta}$, and apply H\"older's inequality.
Since $\phi\in \cI_o(d/R_{w,p}+d/(p\Gamma_w)-d/p,M)$, by \eqref{w scaling} we have
\begin{align}
\label{23.06.12.16.34}
2^{(k-j)(\frac{d}{R_{w,p}}+\frac{d}{p\Gamma_w}-\frac{d}{p}+\varepsilon)}\lesssim\frac{\phi(2^k)}{\phi(2^j)}\lesssim 2^{(k-j)(M-\varepsilon)},\quad k\geq j.
\end{align}
for sufficiently small $\varepsilon>0$.
Then we apply Fubini's theorem and \eqref{23.06.12.16.34} to obtain
\begin{align*}
    \sum_{k\in\bZ}\phi(2^k)^q\sum_{j\leq k}2^{q(j-k)(L-\delta)}\|\Delta_jf\|_{L_p(w)}^q
    &=\sum_{j\in\bZ}\sum_{k=j}^\infty \phi(2^k)^q2^{q(j-k)(L-\delta)}\|\Delta_jf\|_{L_p(w)}^q\\
    &\leq \sum_{j\in\bZ}\sum_{k= j}^\infty \phi(2^j)^q2^{q(j-k)(L-M-\delta+\varepsilon)}\|\Delta_jf\|_{L_p(w)}^q.
\end{align*}
Since $L-M-\delta+\varepsilon>0$ for appropriate $\varepsilon,\delta>0$, $\sum_{k=j}^{\infty}2^{q(j-k)(L-M-\delta+\varepsilon)}\leq N$,
where $N$ is independent of $j$. 
Thus
$$
J_{1,q}\leq N\sum_{j\in\bZ}\phi(2^j)^q\|\Delta_jf\|_{L_p(w)}^q \quad \forall q\in[1,\infty).
$$
For $q=\infty$, by \eqref{23.06.12.16.34},
\begin{align*}
    \phi(2^k)\sum_{j\leq k}2^{(j-k)L}\|\Delta_jf\|_{L_p(w)}   &\lesssim\sum_{j\leq k}\phi(2^j)2^{(j-k)(L-M+\varepsilon)}\|\Delta_jf\|_{L_p(w)}\\
    &\leq\sup_{j\in\bZ}\phi(2^j)\|\Delta_j f\|_{L_p(w)}\sum_{j\leq k}2^{(j-k)(L-M+\varepsilon)}.
\end{align*}
Since $L-M+\varepsilon>0$, $\sum_{j\leq k}2^{(j-k)(L-M+\varepsilon)}\leq N$,
where $N$ is independent of $k$. Thus
$$
J_{1,\infty}\leq N\sup_{j\in\bZ}\phi(2^j)\|\Delta_jf\|_{L_p(w)}.
$$

\vspace{3mm}

\textbf{Case 2.} Estimation of $J_{2,q}$.

For $q\in[1,\infty)$, by H\"older's inequality we have
\begin{align*}
    \Big(\sum_{j\geq k+1} 2^{(j-k)(\frac{d}{r} + \frac{d}{ps}-\frac{d}{p})} \|\Delta_jf\|_{L_p(w)}\Big)^q
    \lesssim \sum_{j\geq k+1}2^{q(j-k)(\frac{d}{r} + \frac{d}{ps}-\frac{d}{p}+\delta)}\|\Delta_jf\|_{L_p(w)}^q.
\end{align*}
Modifying \eqref{23.06.12.16.34}, we have
\begin{align}
\label{23.06.12.16.48}
2^{(j-k)(\frac{d}{R_{w,p}}+\frac{d}{p\Gamma_w}-\frac{d}{p}+\varepsilon)}\lesssim\frac{\phi(2^j)}{\phi(2^k)}\lesssim 2^{(j-k)(M-\varepsilon)},\quad j\geq k.
\end{align}
Then applying Fubini's theorem and \eqref{23.06.12.16.48} yields
\begin{align*}
    &\sum_{k\in\bZ}\phi(2^k)^q\sum_{j\geq k+1}2^{q(j-k)(\frac{d}{r} + \frac{d}{ps}-\frac{d}{p}+\delta)}\|\Delta_jf\|_{L_p(w)}^q\\
    &=\sum_{j\in\bZ}\sum_{k\leq j-1}\phi(2^k)^q 2^{q(j-k)(\frac{d}{r} + \frac{d}{ps}-\frac{d}{p}+\delta)} \|\Delta_jf\|_{L_p(w)}^q\\
    &\leq N\sum_{j\in\bZ}\sum_{k\leq j}\phi(2^{j})^q 2^{q(k-j)(\frac{d}{R_{w,p}} + \frac{d}{p\Gamma_w}-\frac{d}{p} + \varepsilon)} 2^{q(j-k)(\frac{d}{r} + \frac{d}{ps}-\frac{d}{p} + \delta)}\|\Delta_jf\|_{L_p(w)}^q.
\end{align*}
Thus for sufficiently small $\delta>0$ and appropriate $r,s$ satisfying $\varepsilon>\delta+d(\frac{1}{r}-\frac{1}{R_{w,p}}+\frac{1}{p}(\frac{1}{s}-\frac{1}{\Gamma_w}))$, we have
$$
\sum_{k\leq j} 2^{q(j-k)( \delta-\varepsilon+d(\frac{1}{r}-\frac{1}{R_{w,p}}+\frac{1}{p}(\frac{1}{s}-\frac{1}{\Gamma_w})))} \leq N<\infty,
$$
where $N$ is independent of $j$. 
Hence for $q\in[1,\infty)$
$$
J_{2,q}\leq N\sum_{j\in\bZ}\phi(2^j)^q\|\Delta_jf\|_{L_p(w)}^q.
$$

For $q=\infty$, by \eqref{23.06.12.16.48},
\begin{align*}
    &\phi(2^k)\sum_{j \geq k+1} 2^{(j-k)(\frac{d}{r} + \frac{d}{ps} - \frac{d}{p} +\delta)} \|\Delta_jf\|_{L_p(w)}\\
    &= \sum_{j\geq k+1 } \phi(2^k)2^{(j-k)(\frac{d}{r} + \frac{d}{ps} - \frac{d}{p}+\delta)} \|\Delta_jf\|_{L_p(w)}\\
    &\lesssim \sum_{j \geq k+1} \phi(2^j) 2^{(k-j)(\frac{d}{R_{w,p}} + \frac{d}{p\Gamma_w} -\frac{d}{p} + \varepsilon)}2^{(j-k)(\frac{d}{r} + \frac{d}{ps} - \frac{d}{p}+\delta)} \|\Delta_jf\|_{L_p(w)}\\
    &\leq \sup_{j\in\bZ}\phi(2^j)\|\Delta_jf\|_{L_p(w)}\sum_{j\geq k+1} 2^{(j-k)( \delta-\varepsilon+d(\frac{1}{r}-\frac{1}{R_{w,p}}+\frac{1}{p}(\frac{1}{s}-\frac{1}{\Gamma_w})))}
\end{align*}
Since we can choose $\delta,r,s$ satisfying $\varepsilon >\delta+d(\frac{1}{r}-\frac{1}{R_{w,p}}+\frac{1}{p}(\frac{1}{s}-\frac{1}{\Gamma_w}))$, we have
$$
	\sum_{j\geq k+1} 2^{(j-k)( \delta-\varepsilon+d(\frac{1}{r}-\frac{1}{R_{w,p}}+\frac{1}{p}(\frac{1}{s}-\frac{1}{\Gamma_w})))} < N
$$
uniformly in $k$.
Thus it follows that 
$$
	J_{2,\infty}\leq N \sup_{j\in\bZ}\phi(2^j)\|\Delta_jf\|_{L_p(w)}.
$$

\subsubsection{Case of $(X,\cX)=(F,\cF)$}

When $X=F$ and $\cX = \cF$, we make use of Peetre's maximal function.
That is, we apply Lemma~\ref{22.12.28.17.28} $(iv)$ and obtain
\begin{equation}\label{ineq_230315_2117}
\begin{aligned}
	\sup_{|\rho|\leq 2^{-k}}|\cD^L_{\rho}f|
	&\leq N\sum_{j\in\bZ}\sup_{|\rho|\leq 2^{-k}}|\cD_{\rho}^L\Delta_jf|\\
	&\leq N\sum_{j\in\bN}\min(1,2^{(j-k)L})(1+2^{j-k})^{d/r}|M_{j,d/r}^{*}f|.
\end{aligned}
\end{equation}
Comparing \eqref{ineq_230322_1605} and \eqref{ineq_230315_2117}, note that \eqref{ineq_230315_2117} has $(1+2^{j-k})^{d/r}$ instead of $(1+ 2^{j-k})^{\frac{d}{r} + \frac{d}{ps}-\frac{d}{p}}$.
To handle the difference, we assume $\phi$ is of class $\cI_o(d/R_{w,p}, M)$ rather than $\cI_o(d/R_{w,p} + d/(p\Gamma_w) - d/{p}, M)$.
Then, following the argument of $(X,\cX)=(B,\cB)$ case, we obtain the desired result under the assumption on $\phi$.
The lemma is proved.

\mysection{Proof of Proposition~\ref{prop_wBesov_equiv}}\label{sec_prop_equiv}

Let $X=B$ or $F$ so that $X_{p,q}^{\phi}(w) = B_{p,q}^{\phi}(w)$ or $F_{p,q}^{\phi}(w)$.
We show the following equivalence relations:
\begin{eqnarray}
	\|f\|_{X_{p,q}^\phi(\bR^d, w~\mathrm{d}x)} 
	\left\{
	\begin{array}{llll}
	\underset{\text{Lemma~\ref{23.02.08.15.47}}}{\simeq} X_0,\\
	\underset{\text{Lemma~\ref{23.02.08.15.52}}}{\simeq} X_1,\\
	\underset{\text{Lemma~\ref{23.02.08.15.54}}}{\simeq} X_2 \underset{\text{Lemma~\ref{23.02.08.15.55}}}{\simeq} X_3\\
	\underset{\text{Lemma~\ref{23.02.09.14.43}}}{\simeq} X_4, X_5.
	\end{array}
	\right.
\end{eqnarray}
Recall that $\|f\|_{B_{p,q}^\phi(w)}$, $\|f\|_{F_{p,q}^\phi(w)}$ are given by 
 \begin{eqnarray*}
    \|f\|_{B_{p,q}^\phi(w)}&:=&\|S_0f\|_{L_p(w)}
    +\left\{
    \begin{array}{ll}
    \Big(\sum_{j=1}^{\infty}\phi(2^{j})^q\|\Delta_jf\|_{L_p(w)}^q\Big)^{1/q}, &q<\infty,\\
    \sup_{j\in\bN}\phi(2^j)\|\Delta_jf\|_{L_p(w)}, &q=\infty,
    \end{array}
    \right.\\
\|f\|_{F_{p,q}^\phi(w)}&:=&\|S_0f\|_{L_p(w)}
    +\Big\|\Big(\sum_{j=1}^{\infty}\phi(2^{j})^q|\Delta_jf|^q\Big)^{1/q}\Big\|_{L_p(\bR^d,w\,\mathrm{d}x)}.
\end{eqnarray*}

\begin{lem}\label{23.02.08.15.47}
	$ \|f\|_{X_{p,q}^{\phi}(w)}\simeq X_0$.
\end{lem}
\begin{proof}
Recall that 
\begin{align*}
	B_0 &= \|f\|_{L_p(w)} + \Big\| \{ \phi(2^j)^q\|\Delta_jf\|_{L_p(w)}^q \}_{j\in\bZ} \Big\|_{\ell^q(\bZ)},\\
	F_0 &= \|f\|_{L_p(w)} + \Big\| \| \{ \phi(2^j)^q \left| \Delta_jf \right|^q\}_{j\in\bZ} \|_{\ell^q(\bZ)} \Big\|_{L_p(w)},
\end{align*}
when $q\geq1$ including $q=\infty$.
Since $| S_0f (x) |\leq N\cM(f)(x)$, one immediately obtains
$$
\|f\|_{X_{p,q}^{\phi}(w)}\lesssim X_0.
$$
For the converse inequality, observe that
\begin{align}
	\|f\|_{L_p(w)} &\leq  \|S_0f\|_{L_p(w)}+ A_{q'}\Big\| \{\phi(2^j) \|\Delta_jf\|_{L_p(w)}\}_{j\geq1} \Big\|_{\ell^q(\bZ)},\label{ineq_230504_1659}\\
	\|f\|_{L_p(w)} &\leq  \|S_0f\|_{L_p(w)}+ A_{q'}\Big\| \| \{\phi(2^j) |\Delta_jf|\}_{j\geq1} \|_{\ell^q(\bZ)} \Big\|_{L_p(w)},\label{ineq_230504_1700}
\end{align}
where$A_{q'} = (\sum_{j\geq1} \phi(2^j)^{-q'})^{1/q'}$, which is bounded by $(\sum_{j\geq1} 2^{-j\varepsilon q'})^{1/q'}$ for some $\varepsilon>0$.
Note that we have $\|f\|_{L_p(w)} \lesssim \|f\|_{X_{p,q}^{\phi}(w)}$ by \eqref{ineq_230504_1659} and \eqref{ineq_230504_1700}.
Thus it suffices to handle the following terms:
$$
	\Big\| \{ \phi(2^j)^q\|\Delta_jf\|_{L_p(w)}^q \}_{j\in\bZ} \Big\|_{\ell^q(\bZ)},\quad \Big\| \| \{ \phi(2^j)^q \left| \Delta_jf \right|^q\}_{j\in\bZ} \|_{\ell^q(\bZ)} \Big\|_{L_p(w)}.
$$
We only consider $j\leq 0$ parts, since $j\geq1$ parts are already contained in $\|f\|_{X_{p,q}^{\phi}(w)}$.

When $X=B$, we make use of
$$
	\Delta_jf=\Delta_j(S_0+\Delta_1)f \quad \forall j=0,-1,-2,\cdots.
$$
Then together with $\phi(2^j) \lesssim 2^{j(M- \varepsilon)}$, we have for $q\in[1,\infty]$
\begin{align*}
	\Big\| \| \{ \phi(2^j)^q\|\Delta_jf\|_{L_p(w)}\}_{j\leq 0} \Big\|_{\ell^q}
	&\lesssim \| \{ 2^{j(M-\varepsilon)} \}_{j\leq1} \|_{\ell^q} \sup_{j\leq 0}\| \Delta_j f\|_{L_p(w)}\\
	&\lesssim \|S_0f\|_{L_p(w)}+\|\Delta_1f\|_{L_p(w)}.
\end{align*}
This certainly implies that
$$
	X_0 \lesssim \|f\|_{B_{p,q}^{\phi}(w)}.
$$

When $X=F$, we apply Lemma~\ref{23.01.31.17.22}.
Let $T_j:=\Delta_j$ and $f_j:=\phi(2^j)f$ in Lemma~\ref{23.01.31.17.22}. 
Since $|\Delta_jf(x)|\leq N\cM(f)(x)$, one verifies that $\{T_j\}_{j\leq0}$ is linearizable and satisfies \eqref{22.09.13.16.42}. 
Then by Lemma \ref{23.01.31.17.22}, it follows that
\begin{align*}
	\Big\| \Big( \sum_{j\leq 0} \phi(2^j)^q \left| \Delta_jf \right|^q\Big)^{\frac{1}{q}} \Big\|_{L_p(w)}
	&=\Big\| \Big( \sum_{j\leq 0} \left| T_jf_j\right|^q \Big)^{\frac{1}{q}} \Big\|_{L_p(w)}\\
	&\lesssim \Big\| \Big( \sum_{j\leq 0} \left| f_j \right|^q\Big)^{\frac{1}{q}} \Big\|_{L_p(w)}
	\lesssim \|f\|_{L_p(w)},
\end{align*}
where the last inequality follows from $\phi(2^j)\lesssim 2^{j(M - \varepsilon)}$.
The lemma is proved.
\end{proof}

\begin{lem}
\label{23.02.08.15.52}
$ \|f\|_{X_{p,q}^{\phi}(w)}\simeq X_1$.
\end{lem}
\begin{proof}
Recall that 
\begin{align*}
	B_1 &= \inf\Big(\|f\|_{L_p(w)} + \Big(\sum_{j\geq0} \phi(2^j)^q \|f-f_j\|_{L_p(w)}\Big)^{\frac{1}{q}} \Big),\\
	F_1 &=  \inf\Big(\|f\|_{L_p(w)} + \Big\| \Big(\sum_{j\geq0} \phi(2^j)^q |f-f_j|^q \Big)^{\frac{1}{q}} \Big) \Big\|_{L_p(w)},
\end{align*}
where the infimum runs over all sequences of functions $\{f_j\}_{j\geq0} \subset \mathcal{S}'(\bR^d)\cap L_p(w)$ such that
$$
	f = \lim_{j\to\infty}f_j\,\,\text{in}\,\, \mathcal{S}'(\bR^d), \quad \text{and}\quad \supp(\widehat{f}_j) \subset B_{2^{j+1}}(0).
$$

\textbf{Step 1.} $X_1 \lesssim\|f\|_{X_{p,q}^{\phi}(w)}$.

In this step, by the similarity, we only prove for $X=B$.
From the support condition of $\widehat{f_j}$, we have
$$
	f_j:=S_0f+\sum_{k=1}^{j}\Delta_kf,\quad f_0:=S_0f,
$$
and
$$
	f-f_j=\sum_{k=j+1}^{\infty}\Delta_kf,
$$
Then by Minkowski's inequality, it follows that 
$$
	\|f-f_j\|_{L_p(w)}\leq \sum_{k=j+1}^{\infty}\|\Delta_kf\|_{L_p(w)}.
$$
For $q\in[1,\infty)$, we apply H\"older's inequality to obtain
\begin{align*}
    \sum_{j=0}^{\infty}\phi(2^j)^q\|f-f_j\|_{L_p(w)}^q
    &\leq \sum_{j=0}^{\infty}\phi(2^j)^q\left(\sum_{k=j+1}^{\infty}\|\Delta_kf\|_{L_p(w)}\right)^{q}\\
    &\leq N\sum_{j=0}^{\infty}\phi(2^j)^q\sum_{k=j+1}^{\infty}2^{q\delta(k-j-1)}\|\Delta_kf\|_{L_p(w)}^{q}\\
    &\leq N\sum_{j=0}^{\infty}\sum_{k=j+1}^{\infty}\phi(2^k)^q2^{q(\delta-\varepsilon)(k-j-1)}\|\Delta_kf\|_{L_p(w)}^{q},
\end{align*}
where $\delta=0$ if $q=1$ and $\delta>0$ if $q\in(1,\infty)$. If we choose $\varepsilon>\delta\geq0$ and use Fubini's theorem, then
\begin{align*}
    \sum_{j=0}^{\infty}\sum_{k=j+1}^{\infty}\phi(2^k)^q2^{q(\delta-\varepsilon)(k-j-1)}\|\Delta_kf\|_{L_p(w)}^{q}
    &=\sum_{k=1}^{\infty}\sum_{j=0}^{k-1}\phi(2^k)^q2^{q(\delta-\varepsilon)(k-j-1)}\|\Delta_kf\|_{L_p(w)}^{q}\\
    &\leq N\sum_{k=1}^{\infty}\phi(2^k)^q\|\Delta_kf\|_{L_p(w)}^{q}\leq N\|f\|_{B_{p,q}^{\phi}(w)}^q.
\end{align*}
For $q=\infty$, we simply have
\begin{equation*}
\begin{gathered}
	\phi(2^j)\|f-f_j\|_{L_p(w)}
	\leq N\sum_{k=j+1}^{\infty}2^{\varepsilon(j-k)}\phi(2^k)\|\Delta_kf\|_{L_p(w)}
	\leq N\sup_{k\in\bN}\phi(2^k)\|\Delta_kf\|_{L_p(w)}.
\end{gathered}
\end{equation*}
This implies that
$$
	B_1 \leq N\|f\|_{B_{p,q}^{\phi}(w)}.
$$

\textbf{Step 2.} $\|f\|_{X_{p,q}^{\phi}(w)}\lesssim X_1$.

First we consider $X=B$. Similar to the previous case, for $k\geq2$ one can observe that
$$
	\Delta_kf=\sum_{j=-1}^{\infty}\Delta_k(f_{k+j}-f_{k+j-1}).
$$
Since $\Delta_j f$ is less than a constant multiple of $\cM f$, we use the weighted $L_p$ boundedness of the Hardy-Littlewood maximal function so that
\begin{align*}
    \|\Delta_kf\|_{L_p(w)}\leq N\sum_{j=k-1}^{\infty}\|f_{j}-f_{j-1}\|_{L_p(w)}.
\end{align*}
Take $\delta=0$ if $q=1$ and $\delta>0$ if $q\in(1,\infty)$. 
For $q\in[1,\infty)$, we apply H\"older's inequality and Fubini's theorem to obtain
\begin{equation}\label{ineq_230316_1512}
\begin{aligned}
    \sum_{k=2}^{\infty}\phi(2^k)^q\|\Delta_kf\|_{L_p(w)}^q
    &\leq N\sum_{k=2}^{\infty}\phi(2^k)^q\sum_{j=k-1}^{\infty}2^{q\delta(j-k)}\|f_j-f_{j-1}\|_{L_p(w)}^q\\
    &\leq N\sum_{k=2}^{\infty}\sum_{j=k-1}^{\infty}\phi(2^j)^q2^{q(\delta-\varepsilon)(j-k)}\|f_j-f_{j-1}\|_{L_p(w)}^q\\
    &=N\sum_{j=1}^{\infty}\sum_{k=2}^{j+1}\phi(2^j)^q2^{q(\delta-\varepsilon)(j-k)}\|f_j-f_{j-1}\|_{L_p(w)}^q\\
    &=N\sum_{j=1}^{\infty}\phi(2^j)^q\|f_j-f_{j-1}\|_{L_p(w)}^q,
\end{aligned}
\end{equation}
where $\varepsilon>\delta\geq0$. 
For $q=\infty$, we simply have
$$
	\phi(2^k)\|\Delta_kf\|_{L_p(w)}\leq N\sum_{j=k-1}^{\infty}2^{\varepsilon(k-j)}\phi(2^j)\|f_j-f_{j-1}\|_{L_p(w)}\leq N\sup_{j\in\bN}\phi(2^j)\|f_j-f_{j-1}\|_{L_p(w)}.
$$

It is left to control $\|S_0 f\|_{L_p(w)}$ and $\|\Delta_j f\|_{L_p(w)}$.
One can observe that
\begin{align*}
    \|S_0f\|_{L_p(w)}+\|\Delta_1f\|_{L_p(w)}
    &\leq N\|f\|_{L_p(w)}
    	\leq N\|f_0\|_{L_p(w)}+\|f-f_0\|_{L_p(w)}\\
    &\lesssim \|f_0\|_{L_p(w)}+\Big( \sum_{j=0}^\infty \phi(2^j)^q \|f-f_j\|_{L_p(w)}^q \Big)^{\frac{1}{q}}.
\end{align*}
Therefore,
\begin{align*}
    \|f\|_{B_{p,q}^{\phi}(w)}&\lesssim \|f_0\|_{L_p(w)}+\Big( \sum_{j=0}^\infty \phi(2^j)^q \|f-f_j\|_{L_p(w)}^q \Big)^{\frac{1}{q}}
\end{align*}
for all sequences of functions $\{f_j\}_{j=0}^{\infty}$ and this certainly implies that
$$
	\|f\|_{B_{p,q}^{\phi}(w)}\lesssim X_1.
$$

Now, we suppose $X=F$. 
Similar to \eqref{ineq_230316_1512}, for $q\in(1,\infty)$ we have
$$
	\sum_{k=2}^{\infty}\phi(2^k)^q|\Delta_kf|^q\leq N\sum_{j=1}^{\infty}\phi(2^j)^q\sum_{k=2}^{j+1}2^{q(\delta-\varepsilon)(j-k)}|\Delta_k(f_j-f_{j-1})|^q.
$$
Then $|\Delta_k(f_j-f_{j-1})|\leq N\cM(f_j-f_{j-1})$ yields 
\begin{align*}
    \sum_{j=1}^{\infty}\phi(2^j)^q\sum_{k=2}^{j+1}2^{q(\delta-\varepsilon)(j-k)}|\Delta_k(f_j-f_{j-1})|^q
    &\leq N\sum_{j=1}^{\infty}\phi(2^j)^q|\cM(f_j-f_{j-1})|^q\sum_{k=2}^{j+1}2^{q(\delta-\varepsilon)(j-k)}\\
    &= N\sum_{j=1}^{\infty}\phi(2^j)^q|\cM(f_j-f_{j-1})|^q.
\end{align*}
Let $T_j:=\cM$ and $G_j:=\phi(2^j)(f_j-f_{j-1})$ in Lemma~\ref{23.01.31.17.22}. 
Note that $\{T_j\}_{j\in\bN}$ is linearizable and satisfies \eqref{22.09.13.16.42}. 
Thus by Lemma \ref{23.01.31.17.22}, it follows that
\begin{align*}
    \Big\| \Big( \sum_{j\geq 1} \phi(2^j)^q \left| \cM(f_j-f_{j-1}) \right|^q \Big)^{\frac{1}{q}} \Big\|_{L_p(w)}
    &=\Big\| \Big( \sum_{j\geq 1}  \left| T_j G_j \right|^q \Big)^{\frac{1}{q}} \Big\|_{L_p(w)}\\
    &\lesssim \Big\| \Big( \sum_{j\geq 1}  \left| G_j \right|^q \Big)^{\frac{1}{q}} \Big\|_{L_p(w)}\\
    &=  \Big\| \Big( \sum_{j\geq 1} \phi(2^j)^q \left| f_j - f_{j-1} \right|^q \Big)^{\frac{1}{q}} \Big\|_{L_p(w)}
\end{align*}
Hence
\begin{align*}
    \|f\|_{F_{p,q}^{\phi}(w)}&\leq \|S_0f\|_{L_p(w)}+\Big\| \Big( \sum_{j\geq 1} \phi(2^j)^q \left| f_j - f_{j-1} \right|^q \Big)^{\frac{1}{q}} \Big\|_{L_p(w)},
\end{align*}
for all sequences of functions $\{f_j\}_{j=0}^{\infty}$ and this certainly implies that
$
	\|f\|_{F_{p,q}^{\phi}(w)}\lesssim F_1.
$
The lemma is proved.
\end{proof}

\begin{lem}
\label{23.02.08.15.54}
$\|f\|_{X_{p,q}^{\phi}(w)}\simeq X_2$.
\end{lem}
\begin{proof}
Recall that
\begin{align*}
	B_2 &= \|f\|_{L_p(w)} + \Big( \int_0^\infty \phi(t^{-1})^q \left\| \psi_{\frac{1}{t}} \ast f \right\|_{L_p(w)}^q ~\frac{\mathrm{d}t}{t}\Big)^{\frac{1}{q}},\\
	F_2 &= \|f\|_{L_p(w)} + \Big\| \Big( \int_0^\infty \phi(t^{-1})^q \left| \psi_{\frac{1}{t}} \ast f \right|^q ~\frac{\mathrm{d}t}{t}\Big)^{\frac{1}{q}}\Big\|_{L_p(w)}.
\end{align*}

\textbf{Step 1.} $\|f\|_{X_{p,q}^{\phi}(w)}\lesssim X_2$.

Since $|S_0f|\leq N\cM(f)$, it suffices to prove
\begin{align*}
	\Big(\sum_{j\in\bZ} \phi(2^j)^q\|\Delta_jf\|_{L_p(w)}^q \Big)^{\frac{1}{q}}
	&\lesssim \Big( \int_0^\infty \phi(t^{-1})^q \left\| \psi_{\frac{1}{t}} \ast f \right\|_{L_p(w)}^q ~\frac{\mathrm{d}t}{t}\Big)^{\frac{1}{q}}\\
       \Big\| \Big( \sum_{j\in\bZ} \phi(2^j)^q \left| \Delta_jf \right|^q\Big)^{\frac{1}{q}} \Big\|_{L_p(w)}
       &\lesssim \Big\| \Big( \int_0^\infty \phi(t^{-1})^q \left| \psi_{\frac{1}{t}} \ast f \right|^q ~\frac{\mathrm{d}t}{t}\Big)^{\frac{1}{q}}\Big\|_{L_p(w)}.
\end{align*}
First we prove for $X=B$.
To do so, we define 
\begin{align}\label{ineq_230315_1543}
	\Psi_j:=\sum_{i=-1}^{2}\int_{2^{j-1}}^{2^j}\psi_{2^it}\frac{\mathrm{d}t}{t}\quad
	\text{so that} \quad
	\Delta_j f = N \Delta_j(\Psi_j \ast f).
\end{align}
Indeed, \eqref{ineq_230315_1543} follows from $\Delta_jf=\Delta_j \left( \left(\psi_{t/2}+\psi_{t}+\psi_{2t}+\psi_{4t} \right)\ast  f \right)$ for all $t\in[2^{j-1},2^{j})$ and $\int_{2^{j-1}}^{2^j} t^{-1}~\mathrm{d}t = N$.
By making use of \eqref{ineq_230315_1543}, we have for $p\in(1,\infty],q\in[1,\infty)$,
\begin{align*}
    \sum_{j=1}^{\infty}\phi(2^{j})^q\|\Delta_jf\|_{L_p(w)}^q&\leq N\sum_{j=1}^{\infty}\phi(2^{j})^q\|\Psi_j\ast f\|_{L_p(w)}^q\\
    &\leq N\sum_{j=1}^{\infty}\int_{2^{j-1}}^{2^{j}}\phi(2^{j})^q\sum_{i=-1}^2\|\psi_{2^it}\ast f\|_{L_p(w)}^q\frac{\mathrm{d}t}{t}\\
    &\leq N\sum_{i=-1}^2\int_{1}^{\infty}\phi(2^{i}t)^q\|\psi_{2^it}\ast f\|_{L_p(w)}^q\frac{\mathrm{d}t}{t}\\
    &\leq 4N\int_{1}^{\infty}\phi(t)^q\|\psi_t\ast f\|_{L_p(w)}^q\frac{\mathrm{d}t}{t}.
\end{align*}
Note that the last line equals to 
\begin{align}
	4N\int_{0}^{1}\phi(t^{-1})^q\|\psi_{1/t}\ast f\|_{L_p(w)}^q\frac{\mathrm{d}t}{t}.
\end{align}
If $q=\infty$, then
$$
	\phi(2^j)\|\Delta_jf\|_{L_p(w)}=\phi(2^j)\|\psi_{2^j}\ast f\|_{L_p(w)}\leq \sup_{t>0}\phi(t^{-1})\|\psi_{1/t}\ast f\|_{L_p(w)},
$$
and this implies that
$$
	\|f\|_{B_{p,q}^\phi(w)} \lesssim B_2,\quad \forall p\in(1,\infty],q\in[1,\infty].
$$

For $X=F$, we have
\begin{align*}
    \Big( \sum_{j\geq1} \phi(2^j)^q |\Delta_jf|^q \Big)^{\frac{1}{q}} 
    \lesssim \Big( \sum_{j\geq1} \phi(2^j)^q |\cM (\Phi_j \ast f )|^q \Big)^{\frac{1}{q}}.
\end{align*}
By Lemma \ref{23.01.31.17.22},
\begin{align*}
    \Big\| \Big( \sum_{j\geq1} \phi(2^j)^q |\cM (\Phi_j \ast f )|^q \Big)^{\frac{1}{q}}  \Big\|_{L_p(w)}
    \lesssim \Big\| \Big( \sum_{j\geq1} \phi(2^j)^q |\Phi_j \ast f |^q \Big)^{\frac{1}{q}}  \Big\|_{L_p(w)}.
\end{align*}
Note that 
\begin{align*}
    \sum_{j\geq 1} \phi(2^j)^q|\Phi_j\ast f|^q
    &\lesssim \sum_{j\geq 1}\sum_{i=-1}^{2}\int_{2^{j-1}}^{2^{j}}\phi(2^j)^q|\psi_{2^it}\ast f|^q\frac{\mathrm{d}t}{t}\\
    &\leq \sum_{j=1}^{\infty}\sum_{i=-1}^{2}\int_{2^{j-1}}^{2^{j}}\phi(2^it)^q|\psi_{2^it}\ast f|^q\frac{\mathrm{d}t}{t}\\
    &\lesssim \int_{0}^{\infty}\phi(t)^q|\psi_{t}\ast f|^q\frac{\mathrm{d}t}{t}.
\end{align*}
Hence we have
$$
	\|f\|_{F_{p,q}^{\phi}(w)}\lesssim F_2.
$$

\textbf{Step 2.} $X_2 \lesssim \|f\|_{X_{p,q}^{\phi}(w)}$.

By Lemma \ref{Lp_bound}, it suffices to prove
\begin{align*}
	\Big( \int_0^\infty \phi(t^{-1})^q \left\| \psi_{\frac{1}{t}} \ast f \right\|_{L_p(w)}^q ~\frac{\mathrm{d}t}{t}\Big)^{\frac{1}{q}}
	&\lesssim \|f\|_{B_{p,q}^{\phi}(w)},\\
	\Big\| \Big( \int_0^\infty \phi(t^{-1})^q \left| \psi_{\frac{1}{t}} \ast f \right|^q ~\frac{\mathrm{d}t}{t}\Big)^{\frac{1}{q}}\Big\|_{L_p(w)}
	&\lesssim \|f\|_{F_{p,q}^{\phi}(w)}.
\end{align*}
We begin with $X=B$. 
If $t<1/4$, then $\psi_t\ast f=\psi_t\ast S_0f$, hence we have
\begin{align*}
	\int_0^{1/4}\phi(t)^{q}\left\|\psi_t\ast f\right\|_{L_p(w)}^q\frac{\mathrm{d}t}{t}
	&\leq N\|S_0f\|_{L_p(w)}^q\int_0^{1/4}\phi(t)^{q}\frac{\mathrm{d}t}{t}\\
	&=N\|S_0f\|_{L_p(w)}^q,\quad \forall q\in[1,\infty)
\end{align*}
and
$$
	\sup_{t\in(0,1/4)}\phi(t)\|\psi_t\ast f\|_{L_p(w)}\leq \sup_{t\in(0,1/4)}\phi(t)\|S_0f\|_{L_p(w)}\leq N\|S_0f\|_{L_p(w)}.
$$
Then, we make use of the following equality
$$
\psi_t\ast f=\psi_t\ast(\Delta_{j-1}f+\Delta_{j}f+\Delta_{j+1}f+\Delta_{j+2}f),\quad \forall t\in[2^{j-1},2^j), j\geq-1
$$
to obtain that for $q\in[1,\infty)$
\begin{align*}
    \int_{1/4}^{\infty}\phi(t)^q\left\|\psi_t\ast f\right\|_{L_p(w)}^q\frac{\mathrm{d}t}{t}
    &=\sum_{j=-1}^{\infty}\int_{2^{j-1}}^{2^{j}}\phi(t)^q\sum_{i=-1}^2\|\Delta_{j+i}f\|_{L_p(w)}^q\frac{\mathrm{d}t}{t}\\
    &\leq N\|f\|_{B_{p,q}^{\phi}(\bR^d,w,\mathrm{d}x)}^q.
\end{align*}
For $q=\infty$, we have
\begin{align*}
	\phi(t)\|\psi_t\ast f\|_{L_p(w)}
	&\leq \phi(t)\sum_{i=-1}^2\|\Delta_{j+i}f\|_{L_p(w)}\\
	&\leq N\|f\|_{B_{p,\infty}^{\phi}(\bR^d,w,\mathrm{d}x)},\quad \forall t\in[2^{j-1},2^j),j\geq-1.
\end{align*}
This certainly implies that
$$
	B_2 \lesssim \|f\|_{B_{p,q}^\phi(w)},\quad \forall p\in(1,\infty],q\in[1,\infty].
$$

For $X=F$, we proceed in a similar manner.
Note that
$$
\begin{gathered}
|\psi_t\ast f|\leq N\cM(S_0f),\quad \forall t\in(0,1/4),\\
|\psi_t\ast f|\leq N\cM((\Delta_{j-1}+\Delta_{j}+\Delta_{j+1}+\Delta_{j+2})f),\quad \forall t\in[2^{j-1},2^j), j\geq-1,
\end{gathered}
$$
Then for $q\in(1,\infty)$ we have
$$
\int_0^{1/4}\phi(t)^{q}\left|\psi_t\ast f\right|^q\frac{\mathrm{d}t}{t}\leq N\cM(S_0f)
$$
and
\begin{align*}
    \int_{1/4}^{\infty}\phi(t)^{q}\left|\psi_t\ast f\right|^q\frac{\mathrm{d}t}{t}
    &\leq \sum_{j=-1}^{\infty}\int_{2^{j-1}}^{2^j}\phi(t)^q\sum_{i=-1}^{2}|\cM(\Delta_{j+i}f)|^q\frac{\mathrm{d}t}{t}\\
    &\leq N\sum_{j=-2}^{\infty}\phi(2^j)^q|\cM(\Delta_{j}f)|^q.
\end{align*}
With the help of Lemma \ref{23.01.31.17.22},
$$
	F_2 \lesssim \|f\|_{F_{p,q}^{\phi}(w)}.
$$
The lemma is proved.
\end{proof}

\begin{lem}
\label{23.02.08.15.55}
$X_2 \simeq X_3$.
\end{lem}

\begin{proof}
Recall that 
\begin{align*}
	B_3 &= \|f\|_{L_p(w)} + \Big( \int_0^\infty \phi(t^{-1})^q \left\| \psi_{\frac{1}{t}}^L \ast f \right\|_{L_p(w)}^q ~\frac{\mathrm{d}t}{t}\Big)^{\frac{1}{q}},\\
	F_3 &= \|f\|_{L_p(w)} + \Big\| \Big( \int_0^\infty \phi(t^{-1})^q \left| \psi_{\frac{1}{t}}^L \ast f \right|^q ~\frac{\mathrm{d}t}{t}\Big)^{\frac{1}{q}}\Big\|_{L_p(w)},
\end{align*}
where 
\begin{align}\label{2307101129}
    \psi_{\frac{1}{t}}^L(y) = \sum_{j=1}^L \binom{L}{j} (-1)^{L-j} \psi_{\frac{1}{jt}}(y).
\end{align}

\begin{enumerate}
\item[$B_2\lesssim B_3$:]
Observe that
$$
	\sum_{k\in\bZ}\widehat{\psi}_{t}^L (2^{k}\xi) = \sum_{j=1}^{L} \binom{L}{j}(-1)^{L-j}=(-1)^{L+1},
$$
and
$$
	\supp(\widehat{\psi}_t)\subseteq \supp(\widehat{\psi}^L_t) = \{\xi\in\bR^d:(2L)^{-1}t\leq|\xi|\leq 2t\}.
$$
Thus we have
\begin{align}\label{ineq_230317_1501}
	\psi_{\frac{1}{t}}\ast f=(-1)^{L+1}\psi_{\frac{1}{t}}\ast(\psi_{\frac{2}{t}}^L + \psi_{\frac{1}{t}}^{L}+\psi_{\frac{1}{2t}}^{L})\ast f.
\end{align}
By Lemma \ref{22.12.28.17.28}, we also have
\begin{align}\label{ineq_230316_1609}
	|\psi_{\frac{1}{t}}\ast f|
	\lesssim \left(\left|\cM(\psi_{\frac{2}{t}}^{L}\ast f)\right|+ \left|\cM(\psi_{\frac{1}{t}}^{L}\ast f)\right|+\left|\cM(\psi_{\frac{1}{2t}}^{L}\ast f)\right|\right).
\end{align}
Using \eqref{ineq_230316_1609}, one can check that for $p\in(1,\infty]$ and $q\in[1,\infty]$
\begin{align*}
    \Big( \int_0^\infty \phi(t^{-1})^q \left\|\psi_{\frac{1}{t}}\ast f\right\|_{L_p(w)}^q ~\frac{\mathrm{d}t}{t} \Big)^{\frac{1}{q}}
    \lesssim \Big( \int_0^\infty \phi(t^{-1})^q \left\|\psi_{\frac{1}{t}}^L \ast f\right\|_{L_p(w)}^q ~\frac{\mathrm{d}t}{t} \Big)^{\frac{1}{q}}.
\end{align*}
This proves 
$$
	B_2 \lesssim B_3.
$$

\item[$F_2\lesssim F_3$:]

Note that for $2^{j-1} <t<2^j$
\begin{align}
	\widehat{\psi}_\frac{1}{t} 
	&= (-1)^{L+1} \widehat{\psi}_\frac{1}{t} \left( \widehat{\psi}_{2^{-j-2}}^L + \widehat{\psi}_{2^{-j-1}}^L +\widehat{\psi}_{2^{-j}}^L+\widehat{\psi}_{2^{-j+1}}^L + \widehat{\psi}_{2^{-j+2}}^L\right),\label{ineq_230317_1502}\\
	\psi_{2^j}^L &= N \psi_{2^j}^L \ast \left(\sum_{i=-2}^{2} \int_{2^{j-1}}^{2^j} \psi_{2^i t}^L ~\frac{\mathrm{d}t}{t}\right).\label{ineq_230317_1503}
\end{align}
The inequality \eqref{ineq_230317_1502} is nothing but \eqref{ineq_230317_1501}, 
and the inequality \eqref{ineq_230317_1503} is a variant of \eqref{ineq_230315_1543}.
Then by making use of \eqref{ineq_230317_1502}, we have
\begin{align*}
	\int_0^\infty \phi(t^{-1})^q | \psi_\frac{1}{t} \ast f|^q ~\frac{\mathrm{d}t}{t}
	&\lesssim \sum_{j\in\bZ} \phi(2^{-j})^q \left( \cM(\psi_{-j}^L \ast f)(x) \right)^q.
\end{align*}
Since $\cM$ is bounded on $L_p(w)$, we apply Lemma~\ref{23.01.31.17.22} and obtain
\begin{equation}\label{ineq_230317_1511}
\begin{aligned}
	\Big\| \Big( \int_0^\infty \phi(t^{-1})^q | \psi_\frac{1}{t} \ast f|^q ~\frac{\mathrm{d}t}{t} \Big)^{\frac{1}{q}} \Big\|_{L_p(w)}
	&\lesssim  \Big\| \Big( \sum_{j\in\bZ} \phi(2^{-j})^q \left|\psi_{2^{-j}}^L \ast f\right|^q \Big)^\frac{1}{q} \Big\|_{L_p(w)}\\
	&= \Big\| \Big( \sum_{j\in\bZ} \phi(2^{j})^q \left|\psi_{2^{j}}^L \ast f\right|^q \Big)^\frac{1}{q} \Big\|_{L_p(w)}.
\end{aligned}
\end{equation}
Then we apply \eqref{ineq_230317_1503} to the last line of \eqref{ineq_230317_1511}.
\begin{equation}\label{ineq_230317_1517}
\begin{aligned}
	 &\Big\| \Big( \sum_{j\in\bZ} \phi(2^{j})^q \left|\psi_{2^{j}}^L \ast f\right|^q \Big)^\frac{1}{q} \Big\|_{L_p(w)}\\
	 &= \Big\| \Big( \sum_{j\in\bZ} \phi(2^{j})^q \Big|  \psi_{2^{j}
  }^L \ast \sum_{-2\leq i\leq 2} \int_{2^{j-1}}^{2^j} \psi_{2^i t}^L\ast f ~\frac{\mathrm{d}t}{t} \Big|^q \Big)^\frac{1}{q} \Big\|_{L_p(w)}\\
	 &\lesssim  \Big\| \Big( \sum_{j\in\bZ} \phi(2^{j})^q \Big|  \cM\Big(  \int_{2^{j-1}}^{2^j} \psi_{t}^L\ast f ~\frac{\mathrm{d}t}{t}\Big) \Big|^q \Big)^\frac{1}{q} \Big\|_{L_p(w)}.
  \end{aligned}
\end{equation}
By Lemma~\ref{23.01.31.17.22} again to the last line of \eqref{ineq_230317_1517}, it follows that
\begin{equation}\label{ineq_230317_1525}
\begin{aligned}
	&\Big\| \Big( \sum_{j\in\bZ} \phi(2^{j})^q \Big|  \cM\Big(  \int_{2^{j-1}}^{2^j} \psi_{t}^L\ast f ~\frac{\mathrm{d}t}{t}\Big) \Big|^q \Big)^\frac{1}{q} \Big\|_{L_p(w)}\\
	 &\lesssim \Big\| \Big( \sum_{j\in\bZ} \phi(2^{j})^q \Big|  \int_{2^{j-1}}^{2^j} \psi_{t}^L\ast f ~\frac{\mathrm{d}t}{t}\Big|^q \Big)^\frac{1}{q} \Big\|_{L_p(w)}\\
	 &\leq \Big\| \Big( \sum_{j\in\bZ} \int_{2^{j-1}}^{2^j} \phi(t)^q  \left| \psi_{t}^L\ast f \right|^q~\frac{\mathrm{d}t}{t} \Big)^\frac{1}{q} \Big\|_{L_p(w)}\\
	 &=  \Big\| \Big( \int_0^\infty \phi(t^{-1})^q  \left| \psi_{\frac{1}{t}}^L\ast f \right|^q~\frac{\mathrm{d}t}{t} \Big)^\frac{1}{q} \Big\|_{L_p(w)},
\end{aligned}
\end{equation}
where the last equality follows from summation over $j\in\bZ$ and the change of variable $t\to t^{-1}$.
Combining \eqref{ineq_230317_1511}, \eqref{ineq_230317_1517}, \eqref{ineq_230317_1525}, we conclude
$$
	F_2 \lesssim F_3.
$$
\end{enumerate}

The converse direction, verifying $X_3 \lesssim X_2$, is rather simpler than showing $X_2 \lesssim X_3$.
Let $|f|_X$ denote $|f|_B$, $|f|_F$ with $|f|_B = \|f\|_{L_p(w)}$ and  $|f|_F = |f(x)|$.
Then we have
\begin{align*}
\int_0^\infty \phi(t^{-1})^q |\psi_{\frac{1}{t}}^L \ast f |_X^q ~\frac{\mathrm{d}t}{t}
&\leq \sum_{j=1}^L \binom{L}{j} \int_0^\infty \phi(t^{-1})^q |\psi_{\frac{1}{jt}} \ast f |_X^q ~\frac{\mathrm{d}t}{t}\\
&\lesssim_L \int_0^\infty \phi(t^{-1})^q |\psi_{\frac{1}{t}} \ast f |_X^q ~\frac{\mathrm{d}t}{t}.
\end{align*}
Hence, we easily obtain $X_3 \lesssim X_2$.
The lemma is proved.

\end{proof}

\begin{lem}
\label{23.02.09.14.43}
$\|f\|_{X_{p,q}^{\phi}(w)}\simeq X_4, X_5$.
\end{lem}
\begin{proof}
Due to similarity, we only prove for $X_4$.
Recall that
\begin{align}
	B_4 &=  \|M_{d/r}^{0,*}f\|_{L_p(w)} + \left( \sum_{j\geq1} \phi(2^j)^q \| M_{j,d/r}^{*}f \|_{L_p(w)}^q\right)^{1/q},\\
	F_4 &= \|M_{d/r}^{0,*}f\|_{L_p(w)} + \left\| \left( \sum_{j\geq1} \phi(2^j)^q | M_{j,d/r}^{*}f |^q\right)^{1/q} \right\|_{L_p(w)}.
\end{align}
By Lemma \ref{22.12.28.17.28} ($2$),
$$
	|\Delta_jf(x)|\leq M_{j,d/r}^{*}f(x),\quad |S_0f(x)|\leq M_{d/r}^{0,*}f(x)\quad \text{for all $x\in\bR^d$}
$$
and this certainly implies that
$$
	\|f\|_{X_{p,q}^{\phi}(w)}\lesssim X_4.
$$

For the converse inequality, note that $w\in A_p(\bR^d)$, for $r\in(0,R_{w,p})$, the weight $w\in A_{p/r}(\bR^d)$. 
Then we apply Lemma \ref{22.12.28.17.28} ($3$) to obtain
\begin{align*}
    \|f\|_{X_{p,q}^{\phi,4}(w)}
    \lesssim &\|\cM(|S_0f|^{r})^{1/r}\|_{L_p(w)}
    \\
    &\quad+\left\{ 
    \begin{array}{ll}
        \left( \sum_{j\geq1} \phi(2^j)^q \| \cM(|\Delta_j f|^r)^{\frac{1}{r}} \|_{L_p(w)}^q\right)^{1/q}, &X=B\\
        \left\| \left( \sum_{j\geq1} \phi(2^j)^q | \cM(|\Delta_j f|^r)^{\frac{1}{r}} |^q\right)^{1/q} \right\|_{L_p(w)}, &X=F.
    \end{array}
    \right.
\end{align*}
Thus we have 
$$
	X_4 \lesssim \|f\|_{X_{p,q}^{\phi}(w)}.
$$
The lemma is proved.
\end{proof}

Finally, we prove Theorem~\ref{homo} and Proposition~\ref{poly}.
\begin{proof}[Proof of Theorem \ref{homo}]
By similarity, we only provide proof for the homogeneous Besov case.
It directly follows from the proof of Lemmas~\ref{23.02.08.15.54} and \ref{23.02.08.15.55} that
$$
\|f\|_{\mathring{B}_{p,q}^\phi(\bR^d,w\dd x)}^q\simeq  \int_0^\infty \phi(t^{-1})^q \left\| \psi_{\frac{1}{t}} \ast f \right\|_{L_p(w)}^q ~\frac{\mathrm{d}t}{t}\simeq  \int_0^\infty \phi(t^{-1})^q \left\| \psi^L_{\frac{1}{t}} \ast f \right\|_{L_p(w)}^q ~\frac{\mathrm{d}t}{t}\,,
$$
where $\psi_{1/t}^L$ is the function defined in \eqref{2307101129}.
Due to the proof of \eqref{230611958}, we have
\begin{align*}
 \int_0^\infty \phi(t^{-1})^q \left\| \psi^L_{\frac{1}{t}} \ast f \right\|_{L_p(w)}^q ~\frac{\mathrm{d}t}{t}\,&\lesssim \int_{\bR^d}\bigg(\int_{0}^{\infty}\phi(t^{-1})^q|\psi_{1/t}(y)|\frac{\mathrm{d}t}{t}\bigg)\big\|\cD_y^Lf\big\|_{L_p(w)}^q\dd y\\
&\lesssim \int_{\bR^d}\phi(|y|^{-1})^q\big\|\cD_y^Lf\big\|_{L_p(w)}^q\frac{\dd y}{|y|^d}=\|f\|_{\mathring{\cB}_{p,q}^\phi(\bR^d,w\dd x)}^q.
\end{align*}
Therefore, the estimate
$\|f\|_{\mathring{B}_{p,q}^\phi(w)}\lesssim \|f\|_{\mathring{\cB}_{p,q}^\phi(w)}$
is obtained.
Due to \eqref{2307101154}, 
$$
\|f\|_{\mathring{\cB}_0}\lesssim \|f\|_{\mathring{B}_{p,q}^\phi(w)}\,,
$$
where 
$$
\|f\|_{\mathring{\cB}_0}:=
\begin{cases}
    \left(\int_{\bR^d}\phi(|h|^{-1})^q\sup_{|\rho|\leq|h|}\|\cD^L_{\rho}f\|_{L_p(w)}^q\frac{\mathrm{d}h}{|h|^d}\right)^{1/q},\quad &q<\infty,\\
    \sup_{h\in\bR^d}\phi(|h|^{-1})\sup_{|\rho|\leq|h|}\|\cD^L_{\rho}f\|_{L_p(w)},\quad &q=\infty.
\end{cases}
$$
Since
$$
\|f\|_{\mathring{\cB}_{p,q}^\phi(w)}\lesssim \|f\|_{\mathring{\cB}_0},
$$
we also have the converse estimate.
The theorem is proved.
\end{proof}

We end this section with the proof of Proposition \ref{poly}.

\begin{proof}[Proof of Proposition \ref{poly}]
We first observe that for any multi-index $\alpha\in\big(\bN_0\big)^d$ and $x,\,h\in\bR^d$, 
\begin{align}\label{230714604}
\cD^L_h(x^\alpha):=\sum_{k=0}^{L}\binom{L}{k}(-1)^{L-k}(x+kh)^\alpha=
\begin{cases}
0 &\text{if}\quad |\alpha|< L\\[2mm]
\sum\limits_{\substack{\beta:L\leq |\beta|,\,\beta\leq \alpha}}c_{\alpha,\beta}x^{\alpha-\beta}h^\beta &\text{if}\quad |\alpha|\geq L\,,
\end{cases}
\end{align}
for some nonzero constants $c_{\alpha,\beta}$.
Here, for $\alpha=(\alpha_1,\ldots,\alpha_d)$ and $\beta=(\beta_1,\ldots,\beta_d)$, $\beta\leq \alpha$ denotes that $\beta_i\leq \alpha_i$ for all $i=1,\,\ldots,\,d$.
Indeed,
$$
\cD_h^L(x^\alpha)=\sum_{\beta:\beta\leq \alpha}\left(\sum_{k=0}^L\binom{L}{k}(-1)^{L-k}k^{|\beta|}\right)\binom{\alpha}{\beta}x^{\alpha-\beta}h^\beta,\quad \text{where}\,\,\,\,\binom{\alpha}{\beta}:=\prod_{i=1}^d\binom{\alpha_i}{\beta_i}\,,
$$
and
\begin{align*}
\sum_{k=0}^L\binom{L}{k}(-1)^{L-k}k^{n}=\left[\left(\frac{\mathrm{d}}{\mathrm{d}t}\right)^n\left(\mathrm{e}^t-1\right)^L\right]_{t=0}=
\begin{cases}
    0 &\text{if}\quad n<L\\[2mm]
    \text{positive constant}&\text{if}\quad n\geq L\,.
\end{cases}
\end{align*}
It follows from \eqref{230714604} that if $P$ is a polynomial with $\mathrm{deg}\,P<L$, then $\cD_h^LP\equiv 0$ for all $h\in\bR^d$, which implies the first assertion in \eqref{230714638}.
For the second assertion in \eqref{230714638}, due to similarity, we consider only the case $\mathring{\mathscr{X}}_{p,q}^\phi(\bR^d)=\mathring{\mathscr{B}}_{p,q}^\phi(\bR^d)$.
Let $\mathrm{deg}\,P=:M\geq L$, then \eqref{230714604} implies that
$$
\cD_h^LP(x)=\sum_{i=L}^{M}Q_i(x,h)
$$
where $Q_i(x,h)$ is a polynomial satisfying that
$$
Q_i(x,h)=|h|^iQ_i\big(x,h/|h|\big)\quad\forall\,\, x,\,h\in\bR^d\qquad ;\qquad Q_M\not\equiv 0\,.
$$
Take $(x_0,\sigma_0)\in\bR^d\times \bS^{d-1}$ and $\delta>0$ such that
$$
\inf_{(x,\sigma)\in E_1\times E_2 }|Q_M(x,\sigma)|=:2\epsilon_0>0\,,
$$
where $E_1=B_\delta(x_0)$ and $E_2=B_\delta(\sigma_0)\cap \bS^{d-1}$.
Put 
$$
N_0=\sup_{(x,\sigma)\in E_1\times E_2}\sum_{i=L}^{M-1}|Q_i(x,\sigma)|\quad\text{and}\quad R=:\frac{N_0}{\epsilon_0}\vee 1\,.
$$
Then for any $x\in E_1$, $\sigma\in E_2$, and $r\geq R$,
$$
|\cD_{r\sigma}^{L}P(x)|\geq r^M|Q_M(x,\sigma)|-\sum_{i=L}^{M-1}r^i|Q_i(x,\sigma)|\geq r^{M}\bigg(2\epsilon_0-\frac{N_0}{R}\bigg)\geq \epsilon_0 r^M\,.
$$
Consequently, we obtain that 
\begin{equation}\label{230714725}
\begin{aligned}
\|P\|_{\mathring{\mathscr{B}}_{p,q}^{\phi}(\bR^d,w\,\mathrm{d}x)}&:=\left\|\phi(|h|^{-1})\left\|\cD_h^LP(x)\right\|_{L_p(w)}\right\|_{L_q\left(\bR^d;\frac{\mathrm{d}h}{|h|^d}\right)}\\
&\geq \left\|\phi(|h|^{-1})\left\|\cD_{r\sigma}^LP(x)\right\|_{L_p(E_1;w\mathrm{d}x)}\right\|_{L_q\left(E_2\times [R,\infty);\frac{\mathrm{d}h}{|h|^d}\right)}\\
&\gtrsim\left(w(E_1)\right)^{1/p} \left\|\phi(r^{-1})r^M\right\|_{L_q\left([R,\infty);\frac{\mathrm{d}r}{r}\right)}\,.
\end{aligned}
\end{equation}
Since $\phi\in\cI_o(0,L)$, there exists $\epsilon>0$ such that for any $r\in [R,\infty)$, $r^M\phi(r^{-1})\geq r^L\phi(r^{-1})\gtrsim r^{\epsilon}$ (see \eqref{w scaling}).
In addition, $w(E_1)>0$ because $w>0$ a.e. on $\bR^d$.
Therefore, the last term in \eqref{230714725} is infinite.
The proof is completed.
\end{proof}

\mysection{Applications}\label{sec_apps}
In this section, we explore the applications of our main results.
We begin by introducing a notation.
The space $H_p^{\gamma}(w)$ denotes the set of tempered distributions $f$ satisfying $(1-\Delta)^{\gamma/2}f\in L_p(w)$. When $w\equiv 1$, we omit the notation $w$.

\subsection{Regularity theory for evolution equations}
Our results broaden the understanding of weighted Sobolev regularity theory for evolution equations as providing characterizations on optimal trace spaces.
In this subsection, we offer some remarks into the weighted $L_p$-regularity theory for non-divergence form equations with initial data, as discussed in the research by \cite{CKL2023regularity}, \cite{Dong_Kim2021}, and \cite{Dong_Liu2022}.
Specifically, we focus on the equation:
\begin{equation}
\label{eq:non_divergence}
\partial_t^{\alpha}u(t,x)=a^{ij}(t,x)u_{x^ix^j}(t,x), \quad (t,x)\in(0,T)\times\bR^d,
\end{equation}
Here, $\partial_t^{\alpha}$ is the Caputo fractional derivative of order $\alpha$.
When $\alpha=1$ and $a^{ij}(t,x)=a^{ij}(t)$, J.-H. Choi, I. Kim and J. B. Lee \cite{CKL2023regularity} studied the regularity theory for equation \eqref{eq:non_divergence} with $B$-spaces.
With help of Theorem \ref{thm_lp_diff_equiv}, we also have the regularity theory for equation \eqref{eq:non_divergence} with $\mathscr{B}$-spaces:
\begin{thm}
If $\gamma\in[0,\infty)$, $p,q\in(1,\infty)$, $w\in A_p(\bR^d)$, $\tilde{w}\in A_{q}(\bR)$, and
$$
\phi(\lambda):=\lambda^{\gamma+2}\left(\int_{0}^{\lambda^{-2}}\tilde{w}(t)\mathrm{d}t\right)^{1/q}\in\cI_o\left(\frac{d}{p}\left(R_w+\frac{1}{\Gamma_w}-1\right),M\right)
$$
for some positive integer $M > \frac{d}{p}\left(R_w+\frac{1}{\Gamma_w}-1\right)$, then for a given initial data $u_0\in\mathscr{B}_{p,q}^{\phi}(w)$, there exists a unique solution $u\in L_q((0,T),\tilde{w}\,\mathrm{d}t;H_p^{\gamma+2}(w))$ to the equation
\begin{align*}
\begin{cases}
\partial_tu(t,x)=\Delta u(t,x),\quad &(t,x)\in(0,T)\times\bR^d,\\
u(0,x)=u_0(x),\quad &x\in\bR^d
\end{cases}
\end{align*}
with the estimate
\begin{equation*} \left(\int_{0}^T\left(\|u(t,\cdot)\|_{H_p^{\gamma}(w)}+\|u_{xx}(t,\cdot)\|_{H_p^{\gamma}(w)}+\|\partial_tu\|_{H_p^{\gamma}(w)}\right)^q\tilde{w}(t)\mathrm{d}t\right)^{1/q}\lesssim\|u_0\|_{\mathscr{B}_{p,q}^{\phi}(w)}.
\end{equation*}
Moreover,
\begin{equation*} \left(\int_{0}^T\left(\|u_{xx}(t,\cdot)\|_{H_p^{\gamma}(w)}+\|\partial_tu\|_{H_p^{\gamma}(w)}\right)^q\tilde{w}(t)\mathrm{d}t\right)^{1/q}\lesssim\|u_0\|_{\mathring{\mathscr{B}}_{p,q}^{\phi}(w)}.
\end{equation*}
\end{thm}

The works \cite{Dong_Kim2021,Dong_Liu2022} mainly develop a unified weighted Sobolev regularity theory for parabolic equations with (partially) $VMO$-type coefficients under the zero-initial condition, and briefly also address the nonzero-initial condition with power-type temporal weights.
However, the initial data spaces introduced in \cite{Dong_Kim2021,Dong_Liu2022} are strictly smaller than the trace spaces associated with \eqref{eq:non_divergence}, as we shall see in the next paragraph. 
This shows that their regularity result does not fully encompass all admissible initial data and, therefore, is not optimal from the viewpoint of trace theory.
By employing our new characterizations of weighted Besov spaces, we precisely identify the corresponding trace spaces and thereby establish an optimal weighted Sobolev regularity theory that both extends and sharpens the previous results.

To explain this in detail, we first recall the results on nonzero-initial data from \cite{Dong_Kim2021,Dong_Liu2022}.
For $\alpha<1$, \cite[Lemma 5.7]{Dong_Kim2021} shows that under certain conditions (where $p,q\in(1,\infty)$, $\mu\in(-1,q-1)$, $\alpha\in(\frac{1+\mu}{q},1]$, $w\in A_p(\bR^d)$, and $2-2(1+\mu)/(\alpha q)\in(0,2)\setminus\{1\}$), if the initial data $u_0$ satisfies
\begin{equation}
\label{eq:init_data_condition}
\begin{aligned}
&\|u_0\|_{\mathrm{B}_{p,q}^{2-\frac{2(1+\mu)}{\alpha q}}(\mathbb{R}^d,w\,\mathrm{d}x)}\\
:=&\sum_{|\beta|\leq k}\|D^{\beta}u_0\|_{L_p(w)}+\left(\int_{\bR^d}\frac{\|D^{k}u_0(\cdot+h)-D^{k}u_0\|_{L_p(w)}^q}{|h|^{d+q(2-2(1+\mu)/(\alpha q)-k)}}\mathrm{d}h\right)^{1/q}<\infty,
\end{aligned}
\end{equation}
where $k:=\left\lfloor2-2(1+\mu)/(\alpha q)\right\rfloor$,
then there exists a unique solution $u\in L_q((0,T),t^{\mu}\,\mathrm{d}t;H_p^2(w))$ to \eqref{eq:non_divergence} with $\partial_t^{\alpha}u\in L_q((0,T),t^{\mu}\,\mathrm{d}t;L_p(w))$.
Later, \cite[Appendix C]{Dong_Liu2022} employed a similar approach to treat non-zero initial data problems for time-fractional diffusion-wave equations with $\alpha\in(1+\frac{1+\mu}{q},2)$.
It can be verified that
$$
\|u_0\|_{\mathrm{B}_{p,q}^{2-\frac{2(1+\mu)}{\alpha q}}(\mathbb{R}^d,w\,\mathrm{d}x)}=
\begin{cases}
    \|u_0\|_{\mathscr{B}_{p,q}^{2-\frac{2(1+\mu)}{\alpha q}}(\mathbb{R}^d,w\,\mathrm{d}x)}, &\text{if }\,2-\frac{2(1+\mu)}{\alpha q}\in(0,1),\\
    \|u_0\|_{L_p(w)}+\|\nabla u_0\|_{\mathscr{B}_{p,q}^{1-\frac{2(1+\mu)}{\alpha q}}(\mathbb{R}^d,w\,\mathrm{d}x)}, &\text{if }\,2-\frac{2(1+\mu)}{\alpha q}\in(1,2).
\end{cases}
$$

On the other hand, D. Kim and K. Woo \cite{Kim_Woo2023} established the trace theorem for \eqref{eq:non_divergence}, where they defined the trace space as the weighted Besov space:
\begin{equation}
\label{25.11.03.11.51}
\text{Trace space for \eqref{eq:non_divergence}}=B_{p,q}^{\phi}(\bR^d,w\,\mathrm{d}x),\quad \phi(\lambda):=\lambda^{2-\frac{2(1+\mu)}{\alpha q}}.
\end{equation}
According to Remark \ref{rem:mainthm}, if the weight $w\in A_p(\mathbb{R}^d)$ satisfies
\begin{equation}
\label{25.11.03.19.38}
2-\frac{2(1+\mu)}{\alpha q}\leq \frac{d}{p}\left(R_w+\frac{1}{\Gamma_w}-1\right)<1,
\end{equation}
then 
$$
\mathrm{B}_{p,q}^{2-\frac{2(1+\mu)}{\alpha q}}(\mathbb{R}^d,w\,\mathrm{d}x)=\mathscr{B}_{p,q}^{\phi}(\mathbb{R}^d, w\,\mathrm{d}x)\subsetneq B_{p,q}^{\phi}(\mathbb{R}^d,w\,\mathrm{d}x)=\text{Trace space for \eqref{eq:non_divergence}}.
$$
Therefore, the $\mathrm{B}$-type spaces in \eqref{eq:init_data_condition} are not suitable for establishing the optimal weighted Sobolev regularity theory for \eqref{eq:non_divergence} without assuming the algebraic condition \eqref{25.11.03.19.38}.
A similar argument applies to the case $2-\frac{2(1+\mu)}{\alpha q}\in(1,2)$.

Combining \cite{CLSW2023trace, Dong_Kim2021} with Theorem \ref{thm_lp_diff_equiv}, we obtain an optimal weighted Sobolev regularity theory consistent with the setting of \cite{Dong_Kim2021}.
\begin{thm}
\label{23.05.29.13.46}
    Let $\alpha\in(0,1]$, $p,q\in(1,\infty)$, $w\in A_p(\bR^d)$. and $\tilde{w}\in A_q(\bR)$. Suppose that
    $$
    \tilde{W}_{\alpha}(\lambda):=\left(\int_{0}^{\lambda}\tilde{w}(t^{1/\alpha})\mathrm{d}t\right)^{1/q}\in \cI_o\left(0,1\right),
    $$
    \begin{equation}
    \label{25.11.03.19.52}
    \phi(\lambda):=\lambda^2\tilde{W}_{\alpha}(\lambda^{-2})\in\cI_o\left(\frac{d}{p}\left(R_w+\frac{1}{\Gamma_w}-1\right),M\right),
    \end{equation}
    for some positive integer $M > \frac{d}{p}\left(R_w+\frac{1}{\Gamma_w}-1\right)$.
    Then for a given initial data $u_0\in \mathscr{B}_{p,q}^{\phi}(w)$, there exists a unique solution $u\in L_q((0,T),\tilde{w}\,\mathrm{d}t;H_p^2(w))$ to the equation
    \begin{align}
    \label{23.05.29.13.09}
        \begin{cases}
        \partial_t^{\alpha}u(t,x)=\Delta u(t,x),\quad &(t,x)\in(0,T)\times\bR^d,\\
        u(0,x)=u_0(x),\quad &x\in\bR^d
    \end{cases}
    \end{align}
    with the estimate
    \begin{equation}
        \label{23.05.29.13.11}
        \||u|+|u_{xx}|+|\partial_t^{\alpha}u|\|_{L_q((0,T),\tilde{w}\,\mathrm{d}t;L_p(w))}\lesssim\|u_0\|_{\mathscr{B}_{p,q}^{\phi}(w)}.
    \end{equation}
\end{thm}
\begin{proof}
    By \cite[Proposition A.3]{CLSW2023trace},
    $$
    (H_p^2(\bR^d,w\,\mathrm{d}x),L_p(\bR^d,w\,\mathrm{d}x))_{\widetilde{W}_{\alpha},q}=B_{p,q}^{\phi}(\bR^d,w\,\mathrm{d}x),
    $$
    where $(H_p^2(\bR^d,w\,\mathrm{d}x),L_p(\bR^d,w\,\mathrm{d}x))_{\widetilde{W}_{\alpha},q}$ is the generalized interpolation space (for definition, see \cite[Section 3]{CLSW2023trace}).
First, we choose 
    $$
    v\in L_q(\bR_+,\tilde{w}\,\mathrm{d}t;H_p^2(w))\quad \text{and}\quad g\in L_q(\bR_+,\tilde{w}\,\mathrm{d}t;L_p(w))
    $$
    such that
    $$
    \begin{cases}
        \partial_t^{\alpha}v(t,x)=g(t,x),\quad &(t,x)\in \bR_+\times\bR^d,\\
        v(0,x)=u_0(x),\quad &x\in\bR^d,
    \end{cases}
    $$
    with the estimate
    \begin{align}   
    \label{23.08.08.17.17}    \|v\|_{L_q(\bR_+,\tilde{w}\,\mathrm{d}t;H_p^2(w))}+\|g\|_{L_q(\bR_+,\tilde{w}\,\mathrm{d}t;L_p(w))}\lesssim \|u_0\|_{B_{p,q}^{\phi}(\bR^d,w\,\mathrm{d}x)}.
    \end{align}
    The existence of $v$ and $g$ is established in \cite[Theorem 1.4]{CLSW2023trace}. By \cite[Theorem 2.2]{Dong_Kim2021}, there exists a unique solution $\tilde{v}\in L_q((0,T),\tilde{w}\,\mathrm{d}t;H_p^2(w))$ such that
    $$
    \begin{cases}
        \partial_t^{\alpha}\tilde{v}(t,x)=\Delta \tilde{v}(t,x) +(\Delta v(t,x)-g(t,x)),\quad &(t,x)\in (0,T)\times\bR^d,\\
        \tilde{v}(0,x)=0,\quad &x\in\bR^d,
    \end{cases}
    $$
    with the estimate
    \begin{align}
        \label{23.08.08.17.21}
\||\tilde{v}|+|\tilde{v}_{xx}|+|\partial_t^{\alpha}\tilde{v}|\|_{L_q((0,T),\tilde{w}\,\mathrm{d}t;L_p(w))}\lesssim \||\Delta v|+|g|\|_{L_q((0,T),\tilde{w}\,\mathrm{d}t;L_p(w))}.
    \end{align}
    If we put $u:=v+\tilde{v}$, then $u$ is a solution of \eqref{23.05.29.13.09} satisfying \eqref{23.05.29.13.11}.
    Therefore, \eqref{23.08.08.17.17}, \eqref{23.08.08.17.21} and Theorem \ref{thm_lp_diff_equiv} yield that
    \begin{align*}
    \||u|+|u_{xx}|+|\partial_t^{\alpha}u|\|_{L_q((0,T),\tilde{w}\,\mathrm{d}t;L_p(w))}&\lesssim \||\tilde{v}|+|\tilde{v}_{xx}|+|\partial_t^{\alpha}\tilde{v}|\|_{L_q((0,T),\tilde{w}\,\mathrm{d}t;L_p(w))}\\
    &\quad +\|v|+|v_{xx}|+|\partial_t^{\alpha}v|\|_{L_q((0,T),\tilde{w}\,\mathrm{d}t;L_p(w))}\\
        &\lesssim \||v|+|v_{xx}|+|g|\|_{L_q((0,T),\tilde{w}\,\mathrm{d}t;L_p(w))}\\
        &\lesssim \|v\|_{L_q(\bR_+,\tilde{w}\,\mathrm{d}t;H_p^2(w))}+\|g\|_{L_q(\bR_+,\tilde{w}\,\mathrm{d}t;L_p(w))}\\
        &\lesssim \|u_0\|_{B_{p,q}^{\phi}(w)}\simeq\|u_0\|_{\mathscr{B}_{p,q}^{\phi}(w)}.
    \end{align*}
The uniqueness follows from \cite[Theorem 2.2]{Dong_Kim2021}. The theorem is proved.
\end{proof}
In Theorem \ref{23.05.29.13.46}, we have stated the result specifically for the Laplace operator.
However, it can be extended to second-order elliptic operators \cite{Dong_Kim2021} or non-local operators.
Moreover, if one considers $B$-spaces instead of $\mathscr{B}$-spaces in Theorem \ref{23.05.29.13.46}, then the condition \eqref{25.11.03.19.52} is no longer required.

\subsection{Generalized Sobolev embedding theorem without weights}
It is well-known that if $p\gamma>d$ and $\gamma-d/p$ is not an integer, then the classical Sobolev-H\"older embedding theorem (see \textit{e.g.}, \cite[Chapter 13, Section 8, Theorem 1]{krylov2008lectures}) states that
\begin{equation}
\label{classical_embedding}
H_p^{\gamma}\hookrightarrow C^{\gamma-\frac{d}{p}}.
\end{equation}
In this subsection, we derive the generalized version of \eqref{classical_embedding} using Theorem \ref{thm_lp_diff_equiv}.

\begin{thm}
\label{23.03.28.16.27}
    Let $p,q\in[1,\infty)$. Suppose that
    $$
    \tilde{\phi}_p(\lambda):=\frac{\phi(\lambda)}{\lambda^{\frac{d}{p}}}\in\cI_o(0,M),
    $$
    for some $M>0$. Then
$$
\|f\|_{C^{\tilde{\phi}_p}_L}\lesssim \|f\|_{F_{p,q}^{\phi}},
$$
where
\begin{align*}
    \|f\|_{C^{\tilde{\phi}_p}_L}:=\|f\|_{L_{\infty}}+\sup_{|h|\leq1}\tilde{\phi}_p(|h|^{-1})\|\cD_{h}^Lf\|_{L_{\infty}},\quad L\geq M.
\end{align*}
\end{thm}
\begin{proof}
By H\"older's inequality,
\begin{align*}
    |\Delta_jf(x)|&\leq \int_{\bR^d}|\psi_{2^{j-1}}(x-y)+\psi_{2^{j}}(x-y)+\psi_{2^{j+1}}(x-y)||\Delta_jf(y)|\mathrm{d}y\\
    &\leq \sum_{i=-1}^{1}\|\Delta_jf\|_{L_p}\|\psi_{2^{j+i}}\|_{L_{p'}}=C2^{jd/p}\|\Delta_jf\|_{L_p},\quad \forall p\in[1,\infty],
\end{align*}
where $1/p+1/p'=1$. This implies that
$$
\sup_{j\in\bN}\phi(2^{j})2^{-jd/p}\|\Delta_jf\|_{L_{\infty}}\lesssim \left(\sum_{j\in\bN}\phi(2^{j})^{\tilde{q}}\|\Delta_jf\|_{L_p}^{\tilde{q}}\right)^{1/\tilde{q}},\quad \forall p,\tilde{q}\in[1,\infty).
$$
Similarly, we have
$$
\|S_0f\|_{L_{\infty}}\lesssim \|S_0f\|_{L_p},\quad \forall p\in[1,\infty].
$$
Therefore
$$
\|f\|_{B_{\infty,\infty}^{\tilde{\phi}_p}}\lesssim \|f\|_{B_{p,\tilde{q}}^{\phi}}.
$$
Since $\tilde{\phi}_p\in\cI_o(0,M)$, by Theorem \ref{thm_lp_diff_equiv}, we have
$$
\|f\|_{B_{\infty,\infty}^{\tilde{\phi}_p}}\simeq \|f\|_{L_{\infty}}+\sup_{|h|\leq 1}\tilde{\phi}_p(|h|^{-1})\|\cD_{h}^Lf\|_{L_{\infty}(\bR^d)}=\|f\|_{C^{\tilde{\phi}_p}_L}.
$$
The remaining part of the proof focuses on proving
$$
\|f\|_{B_{p,\tilde{q}}^{\phi}}\lesssim \|f\|_{F_{p,q}^{\phi}}
$$
for an appropriate $\tilde{q}$.
We divide the proof into two cases.

$\bullet$ Case 1. $p\geq q$.

By taking $\tilde{q}=p$, we have
\begin{align*}
    \|f\|_{B_{p,p}^{\phi}}=\|f\|_{F_{p,p}^{\phi}}.
\end{align*}
Since
\begin{align*}
    \sum_{i=1}^{\infty}\phi(2^j)^p|\Delta_jf(x)|^p\leq \left(\sum_{i=1}^{\infty}\phi(2^j)^q|\Delta_jf(x)|^q\right)^{p/q},
\end{align*}
it is clear that
$$
\|f\|_{F_{p,p}^{\phi}}\leq \|f\|_{F_{p,q}^{\phi}}.
$$

$\bullet$ Case 2. $1\leq p<q$.

By taking $\tilde{q}=q$ and using Minkowski's inequality, we obtain
$$
\left(\sum_{j=1}^{\infty}\phi(2^j)^q\|\Delta_jf\|_{L_p}^q\right)^{1/q}\leq \left\|\left(\sum_{j=1}^{\infty}\phi(2^j)^q|\Delta_jf|^q\right)^{1/q}\right\|_{L_p},
$$
which implies
$$
\|f\|_{B_{p,q}^{\phi}}\leq \|f\|_{F_{p,q}^{\phi}}.
$$
The theorem is proved.
\end{proof}

\begin{rem}
$(i)$ In Theorem \ref{23.03.28.16.27}, if
$$
\phi(\lambda)=\lambda^{\gamma},
$$
then $\tilde{\phi}_p\in \cI_o(0,M)$ is equivalent to $\gamma-d/p>0$. Since $F_{p,2}^{\phi}=H_p^{\gamma}$ and
$$
C^{\tilde{\phi}_p}_{1}=C^{\gamma-\frac{d}{p}},\quad 0<\gamma-\frac{d}{p}<1,
$$
Therefore, Theorem \ref{23.03.28.16.27} includes \eqref{classical_embedding}.

$(ii)$ If $\phi$ is a Bernstein function satisfying
\begin{align}
\label{23.03.28.17.04}
    c_0\left(\frac{R}{r}\right)^{\delta_0}\leq \frac{\phi(R)}{\phi(r)} \quad \forall 0<r\leq R<\infty,
\end{align}
for some $c_0,\delta_0>0$, then $\phi\in \cI(\delta_0,1)$.
Hence, by Proposition \ref{23.05.25.13.12},
$$
\zeta(\lambda):=\frac{\phi(\lambda^2)^{\gamma/2}}{\lambda^{d/p}}\in \cI(\delta_0\gamma-d/p,\gamma-d/p),\quad \gamma>0.
$$
If $\gamma>\frac{d}{\delta_0p}$,
then by Theorem \ref{23.03.28.16.27} and Proposition \ref{23.05.25.13.12},
$$
(1-\phi(\Delta))^{-\gamma/2}L_p=H_p^{\phi,\gamma}\hookrightarrow C^{\zeta}_L,\quad L>\gamma-\frac{d}{p}.
$$
Since for $|h|\leq1$,
\begin{align*}
    \frac{c_0}{|h|^{\delta_0\gamma-\frac{d}{p}}}\leq\frac{\zeta(|h|^{-1})}{\zeta(1)},
\end{align*}
we have
$$
H_p^{\phi,\gamma}\, \hookrightarrow \, C^{\zeta}_L \, \hookrightarrow \, C^{\delta_0\gamma-d/p}.
$$
\end{rem}

\textbf{Acknowledgements.}
J.-H. Choi has been supported by a KIAS Individual Grant (MG102701) at Korea Institute for Advanced Study.
J. B. Lee has been supported by the NRF grants funded by the Korea government(MSIT) (No.2021R1C1C2008252 and No. 2022R1A4A1018904).
J. Seo has been supported by a KIAS Individual Grant (MG095802) at Korea Institute for Advanced Study.
K. Woo has been supported by the NRF grant funded by the Korea government(MSIT) (No. 2019R1A2C1084683 and No. 2022R1A4A1032094).

\end{document}